\documentclass[10pt, reqno]{amsart}
\usepackage{graphicx, amssymb, amsmath, amsthm}
\numberwithin{equation}{section}
\usepackage[numbers]{natbib}
\usepackage[backref]{hyperref} 
\usepackage{color}
\usepackage{amscd}
\usepackage{tikz}

\newcommand{\supp}{\mathop{\mathrm{supp}}}

\renewcommand{\le}{\leqslant}
\renewcommand{\leq}{\leqslant}
\renewcommand{\ge}{\geqslant}
\renewcommand{\geq}{\geqslant}

\newcommand{\R}{\mathbb{R}}
\newcommand{\C}{\mathbb{C}}

\renewcommand{\S}{\mathbb{S}}
\renewcommand{\P}{\mathbb{P}}
\newcommand{\E}{\mathbb{E}}
\newcommand{\Z}{\mathbb{Z}}
\newcommand{\lin}{\text{lin}}

\newcommand{\N}{\mathbb{N}}

\theoremstyle{plain}
\newtheorem{theorem}{Theorem}
\newtheorem{proposition}[theorem]{Proposition}
\newtheorem{lemma}[theorem]{Lemma}
\newtheorem{corollary}[theorem]{Corollary}

\theoremstyle{definition}

\newtheorem{remark}[theorem]{Remark}

\numberwithin{theorem}{section}

\makeatletter
\newcommand{\Extend}[5]{\ext@arrow0099{\arrowfill@#1#2#3}{#4}{#5}}
\makeatother

\begin{document}
	\title[On the pointwise convergence of NLS flow on $ \S^2 $]{On the pointwise convergence of cubic  Schr\"odinger flow on $\mathbb S^2$}
	
	\author[F. Meng]{Fanfei Meng}
	\address{Fanfei Meng
		\newline \indent  College of Science, China agricultural university, Beijing, \ 100083}
	\email{meng\_fanfei@cau.edu.cn}
	
	\author[Y. Song]{Yilin Song}
	\address{Yilin Song
		\newline \indent The Graduate School of China Academy of Engineering Physics, Beijing, 100088,\  China}
	\email{songyilin21@gscaep.ac.cn}
	
	\author[C. Sun]{Chenmin Sun}
	\address{Chenmin Sun
		\newline \indent CNRS, Universit\'e Paris-Est Cr\'eteil, Laboratoire d’Analyse et de Math\'ematiques appliqu\'ees, UMR 8050 du CNRS, 94010 Cr\'eteil cedex, France}
	\email{chenmin.sun@cnrs.fr}
	
	\author[R. Zhang]{Ruixiao Zhang}
	\address{Ruixiao Zhang
		\newline \indent The Graduate School of China Academy of Engineering Physics, Beijing, 100088,\  China}
	\email{zhangruixiao21@gscaep.ac.cn}
	
	\author[J. Zheng]{Jiqiang Zheng}
	\address{Jiqiang Zheng
		\newline \indent Institute of Applied Physics and Computational Mathematics, Beijing, 100088, China.
		\newline\indent National Key Laboratory of Computational Physics, Beijing, 100088, China}
	\email{zheng\_jiqiang@iapcm.ac.cn, zhengjiqiang@gmail.com}



	\begin{abstract}
		\noindent In this paper, we study the almost everywhere convergence of the cubic nonlinear Schr\"odinger flow to the initial data on $\mathbb S^2$,
		\begin{equation*}
			iu_t + \Delta_g u = |u|^2u, \quad (t,x)\in\R\times \S^2.
		\end{equation*}
		Inspired by the randomization method and the ansatz introduced by Burq, Camps, Sun, and Tzvetkov [Preprint, arXiv:2404.18229], we prove almost sure pointwise convergence almost everywhere for the nonlinear solution at very low regularity. This extends Compaan-Luc\`a-Staffilani [Int. Math. Res. Not. IMRN, (1) (2021), 596--647] to the spherical setting. We also provide a new necessary condition for the associated $L^p$ maximal estimate for the linear Schr\"odinger equation on $\S^2$. More precisely, we show that the $L^p$ maximal estimate fails for $s<\frac{1}{2}-\frac{1}{2p}$ with $p\ge 2$. In the special case $p=3$, our result matches the corresponding range in the $\R^2$ case, up to the endpoint, and improves the previous result of Chen-Duong-Lee-Yan [J. Math. Pures Appl. 163 (2022), 433--449].
	\end{abstract}
	
	\maketitle

	
	\section{Introduction}
	\subsection{Background}
	In this work, we consider the pointwise convergence for solutions to the nonlinear Schr\"odinger equation with initial data on $\mathbb S^2$.
	Let $u(t,x)$ be the solution to the following cubic nonlinear Schr\"odinger equation,
	\begin{equation}\tag{NLS}
		\left\{
		\begin{aligned}
			& iu_{t} + \Delta_g u = \vert u \vert^2 u, \\
			& u(0, x) = \phi(x),
		\end{aligned}
		\right. \qquad (t, x) \in \R \times \S^2,
		\label{NLS}
	\end{equation}
	where $ \Delta_g $ is the Laplace-Beltrami operator on $\S^2$.
	The main problem studied in our paper can be formulated as follows.
	For $\phi\in H^s(\mathbb S^2)$, what is the minimal $s$ such that $u(t,x)\to \phi(x)$ as $t\to0$ for almost every $x\in\mathbb S^2$?
	
	Let us briefly review the progress in the pointwise convergence for the free Schr\"odinger flow.
	For the case of $ \R^d $, this problem was formulated by Carleson in \cite{Carleson1980}:
	Find the minimal $ s \in \R $ such that
	\begin{align*}
		\lim_{t \to 0} e^{it\Delta} \phi(x) \stackrel{\text{a.e.}}{=} \phi(x), \qquad \forall ~ \phi \in H^s(\R^d).
	\end{align*}
	Carleson proved the convergence holds for $ s \geq \frac14 $ in the one-dimensional case. By using harmonic analysis tools, the a.e. convergence of the Schr\"odinger operator can be reduced to proving an $L^p$ maximal-in-time estimate.
	Later, Dahlberg-Kenig \cite{DK} gave a counterexample which reveals that $ s \ge \frac14 $ is necessary in dimension one.
	
	For the higher dimensional Euclidean spaces $\R^d$, the geometry will play an important role and is more delicate than in the 1D case.
	By establishing the smoothing estimate, Sj\"olin \cite{Sjolin} and Vega \cite{Vega} established the pointwise convergence for $ s > \frac12 $ with $ d \geq 2 $. Bourgain \cite{Bourgain1995} showed the pointwise convergence holds for some $s<\frac{1}{2}$. Later, Lee \cite{Lee} improved it to $ s > \frac38 $ for $ d = 2 $ by utilizing the localization argument and the wave-packet decomposition. Meanwhile, Bourgain \cite{Bourgain2016} proved that the necessary condition of pointwise convergence is $ s \geq \frac{d}{2(d+1)} $. Inspired by the recent development in harmonic analysis  in \cite{BG1,Guth1,Guth2}, Du-Guth-Li \cite{DGL} and Du-Zhang \cite{DZ} proved the convergence for $ s > \frac{d}{2(d+1)} $ with $ d \geq 2 $ which is sharp up to the endpoint.
	
	Compared to the Euclidean case, the Carleson problem on the torus is more subtle,
	due to the erratic behavior of the periodic Schr\"odinger kernel, which is linked to analytic number theory.
	For general initial data, Moyua-Vega \cite{MV} showed the convergence for $ s > \frac13 $ in the case of $ \mathbb T $.
	Following Moyua-Vega's argument together with the Strichartz inequality on $\mathbb{T}^d$, thanks to the decoupling theory,  Compaan-Luc\`a-Staffilani \cite{Staffilani} proved the convergence for $ s> \frac{d}{d+2} $ which is the best known result up to now.
	For special initial data $ \phi $ such that
	\begin{align*}
		\phi(x) := \sum_{\vert k \vert \sim N} e^{ik\cdot x},
	\end{align*}
	the Schr\"odinger flow can be regarded as a type of Weyl sum.
	Barron \cite{Barron} investigated the $L^4$ maximal estimate for the  Weyl sum on $ \mathbb T^1 $ based on the Weyl-sum estimates and the circle method.
	For the higher dimensional torus $ \mathbb T^d $, $ d \geq 2 $, Miao-Yuan-Zhao \cite{MYZ} proved the convergence for this special initial data for $s>\frac{d}{2(d+1)}$ which is sharp up to the endpoint.
	They also gave some new upper bounds on the Hausd\"orff dimensions of the sets associated with the large values of the Weyl sums.
	Baker, Chen, and Shparlinski \cite{BCS1,BCS2} considered generalized maximal estimates for Weyl sums of degree $d$.
	Very recently, Chen-Miao-Yuan-Zhao \cite{CMYZ} extended \cite{Barron} to maximal estimates for higher-order Weyl sums with $k\ge 3$ of the form
	\begin{align*}
		S(k, t) = \sum_{|n|\leqslant N} e^{i(n\cdot x+n^kt)}.
	\end{align*}
	For $ k = 3 $, their result is sharp up to the endpoint.
	They also gave counterexamples extending the previous result of \cite{Pierce} to $k\ge 3$.
	
	For a non-flat compact manifold without boundary $M$, Wang-Zhang \cite{Wang-Zhang} showed the convergence for $ s > \frac34 $ when $ d = 2 $ and $ s > \frac9{10} $ when $ d = 3 $. For $M=\mathbb S^2$ or $\mathbb T^2$, they improved the convergence for $ s > \frac12 $.
	For the zonal spherical harmonics, which are the special eigenfunctions of $-\Delta_g$ on $\S^d$ with $d\geq2$, Chen-Duong-Lee-Yan \cite{CDLY} proved the convergence for $ s > \frac13 $ on $ \mathbb S^d $.
	They also show the necessity of $ s > \frac14 $ for all $ d \geq 1 $. Despite a huge gap between $\frac{1}{3}$ and $\frac{1}{2}$ for the general initial data on $\S^2$,
	we construct the new counterexample in Section 3 ($s\geq\frac{1}{2}-\frac{1}{2p}$). If $p=3$, our result can match the same result as $\R^2$.

	For the cubic nonlinear Schr\"odinger equation, one can follow the strategy in \cite{Staffilani} to prove almost everywhere convergence for solutions with initial data in $H^s$ with $s>\max\{s_{\lin},s_c\}$ on certain compact manifolds, where $s_{\lin}$ is the regularity at which the linear pointwise convergence holds and $s_c$ is the threshold for local well-posedness. For example, when $M=\mathbb S^2$, $s_c=\frac{1}{4}$ and $s_{\lin}=\frac{1}{2}$ while for $M=\mathbb T^2$, $s_c=0$ and $s_{\lin}=\frac{1}{2}$. To further lower the regularity, Compaan-Luc\`a-Staffilani \cite{Staffilani} considered the following randomized initial data
	\begin{align*}
		\phi^\omega(x)=\sum_{n\in\Z^d}\frac{g_n^\omega}{\langle n\rangle^{\frac{d}{2}+\alpha}}e^{in\cdot x},\,\alpha>0,
	\end{align*}
	where $g_n^\omega$ are independent complex-valued Gaussian variables.
	Under this randomization, $e^{it\Delta}\phi^\omega(x)\in\bigcap_{s<\alpha}H^s(\mathbb T^d)$ almost surely.
	For the Wick-ordered nonlinearity, $F(u)=\big(|u|^2-2\int_{\mathbb T^d}|u(t,x)|^2{\rm d}x\big)u$, they proved almost sure everywhere convergence for this nonlinearity when $\alpha>\frac{d-1}{2}$.

	

	\subsection{Main result}
	In this paper, we study the probabilistic almost everywhere convergence for the cubic NLS on $\S^2$. 
	Note that if $\phi\in L^2(\mathbb{S}^2)$ in \eqref{NLS}, by taking the gauge transform (where we normalize the area of $\mathbb{S}^2$ to be $1$)
	\[u(t,x)\mapsto u(t,x)e^{-it+2it\|\phi\|_{L^2}^2},\]
	\eqref{NLS} is equivalent to the following Wick-ordered NLS:
	\begin{equation}\label{Wick ordered}
		\left\{
		\begin{aligned}
			& iu_{t} + (\Delta_g-1) u =\mathcal{N}(u)\stackrel{\triangle}{=} : \vert u \vert^2 u :, \\
			& u(0, x) = \phi(x),
		\end{aligned}
		\right. \qquad (t, x) \in \R \times \S^2,
	\end{equation}
	where $ : \vert u \vert^2 u : $ is defined as
	\[: \vert u \vert^2 u : (x) \stackrel{\triangle}{=} \big( \vert u(x) \vert^2 -  2\Vert u \Vert_{L^2(\S^2)}^2  \big) u(x).\]
	Before presenting our main result, we introduce the randomization of initial data.
	
	For a given probability space $(\Omega,\mathcal{F},\mathbb{P})$, we denote by  $\{g_{n,k}^\omega\}_{n\in\N,|k|\leq n}$ a sequence of complex standard Gaussian variables which are i.i.d. Let $(\lambda_n)_{n\geq 0} $ be the spectrum of $ \sqrt{-\Delta_g+1} $, and let $ (\mathbf{b}_{n, k})_{|k|\leq n} $ be an orthonormal basis of the spectral projection $\pi_n$ for the eigenspace $E_n$ of $ \lambda_n^2 $. In particular, the eigenvalues can be given by
	\begin{align*}
		\lambda_n = \sqrt{n^2+n+1}, \qquad \text{for} ~ n=0,1,2,\cdots.
	\end{align*}
	The eigenspace $E_n$ consists of the spherical harmonic functions of degree $n$ and the multiplicity for $\lambda_n^2$ is $2n+1$, i.e. $\dim(E_n)=2n+1$.
	We randomize the initial data of \eqref{Wick ordered} as
	\begin{equation}\label{randomize initial data-intro}
		\phi_{\alpha}^{\omega}(x) := \sum_{n=0}^{\infty} \frac1{\lambda_n^{\alpha }} \sum_{\vert k \vert \le n} g_{n, k}^\omega \mathbf{b}_{n, k}(x).
	\end{equation}
	This randomization induces a Gaussian probability measure $\mu_{\alpha}$ with covariance operator $(-\Delta_g +1)^{-\alpha}$, which does not depend on the choice of orthonormal basis $(\mathbf{b}_{n,k})$.
	It is known that
	$$H^{\alpha-1-}(\mathbb{S}^2):=\bigcap_{s<\alpha-1}H^s(\mathbb{S}^2)$$
	is of full $\mu_{\alpha}$-measure, and $\mu_{\alpha}(H^{\alpha-1}(\mathbb{S}^2))=0$.
	Note that the linear Schr\"odinger flow $e^{it\Delta_g}$ preserves the law $\mu_{\alpha}$.
	\vspace{0.3cm}

	For the pointwise convergence problem, our first result concerns the linear Schr\"odinger flow in both probabilistic and deterministic settings:
	
	\begin{theorem}\label{thm:maxestcount}
		For $p\geq2$ and $I=[0,2\pi]$, the maximal estimate
		\begin{align*}
			\Big\|\sup_{t\in I}\big|e^{it\Delta_g}\phi\big|\Big\|_{L^p(\S^2)}\lesssim \|\phi\|_{H^s(\S^2)}
		\end{align*}
		fails for $s<\frac{1}{2}-\frac{1}{2p}$. 
	\end{theorem}
	For $p=2$, we recover the necessary condition $s>\frac{1}{4}$ proved in \cite{CDLY} for the almost everywhere convergence property of the linear flow. 
	After introducing randomness, the pointwise convergence property can be substantially improved:
	\begin{theorem}\label{thm1}
		Let $\alpha>1$. Then $\mu_{\alpha}$-almost surely,
		\begin{align*}
			\lim_{t\to0}e^{it\Delta_{g}}\phi(x)=\phi(x), \quad \mbox{a.e.} ~ x\in\S^2.
		\end{align*}
	\end{theorem}

	Our second result addresses the same problem for the nonlinear Schr\"odinger flow. 
	Note that for $\alpha>\frac{5}{4}$, $\mu_{\alpha}$-almost every initial datum belongs to $H^{\frac{1}{4}+}(\mathbb{S}^2)$; hence we always have a local flow thanks to \cite{BGT}. 
	Below the regularity threshold $\frac14$, it is also known from \cite{BGT2} that the Cauchy problem \eqref{NLS} on $\mathbb{S}^2$ is ill-posed in the sense that the flow map fails to be uniformly continuous in $H^s(\mathbb{S}^2)$ for any $s<\frac{1}{4}$. 
	Therefore, for $\alpha\leq \frac{5}{4}$, the solutions are constructed by the probability method in \cite{Burq2024}. 
	More precisely, if $\alpha>1$, then for $\mu_{\alpha}$-almost every initial datum $\phi$, the solution constructed in \cite{Burq2024} exists on a time interval $[-T,T]$ with $T=T(\phi)>0$ and belongs to $C^0([-T,T];H^{(\alpha-1)-}(\mathbb{S}^2))$. 
	For these solutions, we have:
	\begin{theorem}\label{thm:2}
		Let $\alpha>1$. For $\mu_{\alpha}$-almost every initial datum $\phi\in H^{(\alpha-1)-}(\mathbb{S}^2)$, the local solution $u(t,x)$ to \eqref{NLS} with initial data $\phi$ satisfies
		\begin{align}
			\lim_{t\rightarrow 0} u(t,x)=\phi(x), \quad \text{a.e.} ~ x\in\mathbb{S}^2.
		\end{align}
	\end{theorem}
	\begin{remark}
		From the current literature (see \cite{Wang-Zhang}), the almost everywhere convergence property for \eqref{NLS} is known for any initial data $\phi\in H^{s}(\mathbb{S}^2)$ with $s>\frac{1}{2}$. See Lemma \ref{maximal-easy} and Remark \ref{rm:maximal-easy}.
	\end{remark}	
	\begin{remark}
		For the linear equation, one can promote the convergence to every $x\in\S^2$ in Theorem \ref{thm1}. 
		However, since the Duhamel term will not gain any regularity compared to the linear case, by Sobolev embedding, we can only prove convergence everywhere for $\alpha>2$ in Theorem \ref{thm:2}. 
		
	\end{remark}

	\begin{remark}
		The framework of this paper may also be useful for studying probabilistic pointwise convergence for nonlinear Schr\"odinger equations with higher algebraic nonlinearities on $\mathbb T^2$, in connection with the ansatz in \cite{DNY-Ann}.
	\end{remark}
	
	
	\begin{remark}
		It is also interesting to address a similar problem for the cubic wave equation on $\mathbb T^3$ on a  compact manifold. We also refer the resolution of such problem on $\Bbb T^3$ by Luc\'a-Merino \cite{Luca}.
	\end{remark}

	Let us discuss the main difference between our work and \cite{Staffilani}.
	In the context of the cubic NLS on the torus $\mathbb T^2$,
	Bourgain \cite{Bourgain1995} found that, up to a gauge transform, the second Picard iteration exhibits a regularization effect almost surely, thanks to the multi-linear smoothing effect and the statistical properties of the randomized initial data. Such an observation motivates the introduction of the re-centering ansatz of Bourgain by decomposing the solution to NLS as
	\begin{align}\label{re-centering}
		u(t):=e^{it\Delta}\phi + w(t),
	\end{align}
	where $w(t)$ is the remainder with almost $\frac12$ order smoother than the initial data $u_0$, measured by the Sobolev regularity.
	Such a good structure allows the authors in \cite{Staffilani} to prove the almost sure almost everywhere convergence of the nonlinear Schr\"odinger flow on $\mathbb{T}^2$. 
	Indeed,
	the pointwise convergence of the linear part $e^{it\Delta}\phi$ follows from the randomization structure and from the fact that the law of the initial data is invariant under the linear evolution. As the remainder can be controlled in a Fourier restriction space $X^{0,\frac{1}{2}+}\hookrightarrow C_t^0H_x^{\frac{1}{2}+}$, one can essentially apply the deterministic maximal function estimate to prove the pointwise convergence of the nonlinear part.
	
	However, this structure is not valid in $\mathbb S^2$.
	In \cite{counterexample}, the authors computed the regularity of the second Picard iteration associated to \eqref{Wick ordered}:
	\begin{align*}
		\mathcal{T}_{\S^2}(t,\phi):=-i\int_{0}^{t}e^{i(t-s)(\Delta_g-1)}\big(:|e^{is(\Delta_g-1)}\phi|^2e^{is(\Delta_g-1)}\phi:\big)\,{\rm d}s.
	\end{align*}
	Using the spectral decomposition, the Wick-ordered nonlinearity can be rewritten as
	\begin{align*}
		\mathcal{N}(u)=\mathcal{N}^1(u)+\mathcal{N}^2(u)+\mathcal{N}^3(u),
	\end{align*}
	where
	\begin{align*}
		\mathcal{N}^1(u) & := \sum_{\substack{n_1,n_2,n_3\\n_1\neq n_2,n_2\neq n_3}}u_{n_1}\overline{u_{n_2}}u_{n_3},\\
		\mathcal{N}^2(u) & := 2\sum_{\substack{n_1\neq n_2}}(|u_{n_2}|^2-\|u_{n_2}\|_{L^2}^2)u_{n_1},\\
		\mathcal{N}^3(u) & := \sum_{n}(|u_{n}|^2-\|u_{n}\|_{L^2}^2)u_{n}.
	\end{align*}
	For $\mathbb T^2$, the second term vanishes and the algebraic property of the eigenfunctions makes the first term convenient to handle. For $\S^2$, the additional term $\mathcal{N}^2(u)$ as well as another resonant high-low-low frequency interaction will prevent the improvement of the regularity as observed in \cite{counterexample}. More precisely,
	the second Picard iteration will not have better regularity than the linear part, in the following sense:
	\begin{theorem}[\cite{counterexample}] Let  $\phi_\alpha^\omega$ be the randomized initial data as in \eqref{randomize initial data-intro}.
		For $t\geq0$ and $\alpha>\frac12$, there exist $N_0\in\N$ and $\eta>0$ such that for all $N\geq N_0$, there holds
		\begin{align*}
			\eta|t|(\ln N)^\frac12\leq\big\|\mathcal{T}_{\S^2}(t,P_{\leq N}\phi_{\alpha}^\omega)\big\|_{L^2(\Omega,H^{\alpha-1}(\S^2))},
		\end{align*}
		where $P_{\leq N}$ denotes the spectral projector on frequencies lower than $N$.
	\end{theorem}
	In this spirit, the structure of the probabilistic solution associated to the general randomized initial data $\phi_{\alpha}^\omega$ cannot be of the same form as \eqref{re-centering}.
	
	Based on these observations, we adapt the newly built ansatz from \cite{Burq2024} to prove our theorem for the general randomized initial data. For the special data, i.e. zonal spherical harmonics, the symmetry makes the problem similar to the one-dimensional periodic setting. The Duhamel term still offers the opportunity to gain regularity. We refer to \cite{Burq2018,CDLY}. After introducing the random averaging operator that is typically used to remove the high-low-low interaction in \cite{Burq2024,DNY-Ann}, one can partially extract the strong resonance in the iteration scheme. The new ansatz has the leading term retaining the main regularity of the initial data, while the remainder term could be controlled in a smoother Fourier-restriction space. The key point is that the leading term in the ansatz still has the same law as the initial data (after the frequency localization) for fixed time, and is also continuous in time. Consequently, it enjoys a nice space-time continuity which allows us to prove its pointwise convergence as $t\rightarrow 0$.

	\subsection*{Acknowledgment}
	J. Zheng was supported by National key R$\&$D program of China: 2021YFA1002500 and NSF grant of China (No. 12271051).
	The authors would like to thank Renato Lucà for some interesting discussion.

	\section{Preliminaries}\label{section-Pre}
	In this section, we recall some $L^p$ eigenfunction estimates
	in both deterministic and probabilistic settings, together with some useful probabilistic estimates. We also give definitions and properties of Fourier-Lebesgue spaces and Bourgain spaces.
	\subsection{Randomization for the initial data}
	Let $ \{ g_{n,k}^\omega \}_{n\in\N,|k|\leq n} $ be a sequence of i.i.d. standard complex-valued Gaussian random variables on a probability space $ (\Omega, \mathcal{F}, \P) $, let $ \lambda_n $ be the spectrum of $ \sqrt{-\Delta_g+1} $, and let $ \mathbf{b}_{n, k} $ be a real eigenfunction of $ -\Delta_g+1 $ associated to $ \lambda_n^2 $.
	Then we randomize the initial data $ u_0 $ of \eqref{NLS} as
	\begin{equation}\label{randomize initial data}
		\phi_{\alpha}^{\omega}(x) := \sum_{n=0}^{\infty} \frac1{\lambda_n^{\alpha }} \sum_{\vert k \vert \le n} g_{n, k}^\omega \mathbf{b}_{n, k}(x).
	\end{equation}
	
	\begin{remark}
		We denote
		\[z_n^{-(\alpha-\frac{1}{2})} := \lambda_n^{-(\alpha-\frac{1}{2})} \bigg( \frac{2n+1}{\lambda_n} \bigg)^\frac12, \quad e_n^{\omega}(x) := \frac1{\sqrt{2n + 1}} \sum_{\vert k \vert \le n}g_{n, k}^\omega\mathbf{b}_{n, k}(x).\]
		and rewrite \eqref{randomize initial data} as
		\begin{equation}\label{randomize initial data 1}
			\phi_{\alpha}^{\omega}(x) = \sum_{n=0}^{\infty} \frac1{z_n^{\alpha - \frac12}} e_n^{\omega}(x).
		\end{equation}
		Since $ \lambda_n = \sqrt{n^2 + n + 1} $,  $\lambda_n$ and $z_n$ enjoy the same asymptotic. Thus, we can replace $z_n$ by $\lambda_n$. By the local Weyl's law \cite{counterexample},
		\begin{align}\label{Local-Weyl}
			\sum_{|k|\leq n} |\mathbf{b}_{n,k}(x)|^2 = \sum_{|k|\leq n} \|\mathbf{b}_{n,k}\|_{L^2(\mathbb S^2)}^2=2n+1, \qquad \forall ~ x \in \S^2,
		\end{align}
		one has $ \phi_{\alpha}^{\omega} \in H^{\alpha - 1 -}(\S^2) $ $ \P $-almost surely, which coincides with the results in \cite{Burq2018, Oh2018} for $ \alpha = \frac32 $. Also, $ e^{it\Delta_g} \phi_\alpha^\omega \in H^{\alpha - 1 -}(\S^2) $ $ \P $-almost surely.
	\end{remark}
	
	\subsection{ $L^p$ eigenfunction estimates on $\S^2$}
	Let $\lambda_n^2$ be the $n$-th eigenvalue associated to $-\Delta_g$ on a compact Riemannian surface $M$ and $\pi_n$ is the spectral projector on the $n$-th eigenspace $E_n$. Thus, we define
	\begin{equation*}
		P_{\le N} := \sum_{0\leq n \le N} \pi_n,\quad P_N := P_{\le N} - P_{\le N/2}.
	\end{equation*}
	Sogge \cite{Sogge} proved the following $L^p$ estimate :
	
	\begin{proposition}[Spectral projector estimate]For a general compact Riemannian surface $M$, there exists $C>0$ such that for all $n\geq1$ and $f\in L^2(M)$, the following estimate holds
		\begin{align*}
			\|\pi_nf\|_{L^p(M)}\leq C\begin{cases}
				\lambda_n^{\frac12\big(\frac12-\frac1p\big)}\|\pi_nf\|_{L^2(M)},&2\leq p\leq6,\\
				\lambda_n^{\frac12-\frac2p}\|\pi_nf\|_{L^2(M)},&6\leq p\leq\infty.
			\end{cases}
		\end{align*}
		Specifically, this estimate is sharp in two-sphere $\S^2$.
	\end{proposition}
	
	In our paper, to achieve the threshold regularity in the pointwise convergence, we need the following probabilistic eigenfunction estimate.
	\begin{lemma}[Probabilistic $L^p$ eigenfunction estimate, \cite{Burq-Lebeau}] \label{Lp eigen}
		There exists $C>0$ such that for all $x\in \S^2$, $n\in\N$, the following estimate holds
		\begin{align*}
			\|e_n^\omega(x)\|_{L_\omega^p(\Omega)}\leq C\sqrt{p},\quad p \in[2,\infty].
		\end{align*}
		Furthermore, for every $R\geq1$, there exists $C_0>c_0>0$,
		\begin{align*}
			\mathbb P\big[\|e_n^\omega(x)\|_{L_\omega^p}>R\big]\leq C_0e^{-c_0R^2}.
		\end{align*}
		
	\end{lemma}
	
	\subsection{Wiener chaos}\label{section-WC}
	
	We also need the following classical probabilistic estimates including Khinchine's inequality and Wiener Chaos.
	
	\begin{lemma}[\cite{Tzvetkov}]\label{lem:concentration}
		Let $(\Omega,\mathcal{F},\mathbb{P})$ be a probabilistic space,	for sequence $ \{ c_n \}_{n=0}^{\infty} \in l^2(\N) $ and i.i.d. Gaussian variables $ \{g_n^\omega \}\in N_{\C}(0, 1) $, the following two statements are equivalent and true:
		\begin{enumerate}
			\item[\rm (i)] {\rm (Khinchin's inequality)}: there exists a constant $ C > 0 $ such that
			\begin{equation}\label{Ki}
				\Big\Vert \sum_{n=0}^{\infty} c_n g_n^\omega \Big\Vert_{L_{\omega}^p(\Omega)} \le C \sqrt{p} \Vert c_n \Vert_{l^2(\N)}, \qquad \forall ~ p \ge 2.
			\end{equation}
			\item[\rm (ii)] {\rm (Large deviation)}: there exists two constants $ C_1, C_2 > 0 $ such that
			\begin{equation}\label{Ld}
				\P \Big( \Big\{ \omega \in \Omega ~ \Big\vert ~ \Big\vert \sum_{n=0}^{\infty} c_ng_n^\omega \Big\vert > \lambda \Big\} \Big) \le C_1 e^{-\frac{C_2 \lambda^2}{\Vert c_n \Vert_{l^2(\N)}^2}}, \qquad \forall ~ \lambda > 0.
			\end{equation}
		\end{enumerate}
	\end{lemma}

	\begin{proposition}[Wiener chaos, \cite{Oh2018, Simon1974}]\label{prop:WC}
		For $ k \ge 1 $, $ n = (n_1, n_2, \cdots, n_k) \in \N^k $, $ c(n) \in \C $, and $ \{g_n^\omega\} \in N_{\C}(0, 1) $ i.i.d., we define
		\[S_k(\omega) := \sum_{n \in \N^k} \Big[ c(n) \prod_{i=1}^{k} g_{n_i}^\omega \Big].\]
		Then
		\begin{equation}\label{WC}
			\Vert S_k \Vert_{L_{\omega}^p(\Omega)} \le (p-1)^{k/2} \Vert S_k \Vert_{L_{\omega}^2(\Omega)}, \qquad \forall ~ p \ge 2.
		\end{equation}
	\end{proposition}
	
	\begin{proof}
		One can use Nelson's hypercontractive bounds to prove it, for more details, see Theorem 1.22 in \cite{Simon1974}.
	\end{proof}
	
	The explicit bound \eqref{WC} in $ k, p $ implies the large deviation estimate.
	
	\begin{corollary}\label{coro:large devietion estimate}
		For $ S_k(\omega) $ as in Proposition \ref{prop:WC}, if $ c(n) \in l^2(\N^k) $, then there exist two constants $ C_1, C_2 > 0 $ such that
		\[\P \Big( \Big\{ \omega \in \Omega ~ \Big\vert ~ \big\vert S_k(\omega) \big\vert > \lambda \Big\} \Big) \le C_1 e^{-\frac{C_2 \lambda^2}{k}}, \qquad \forall ~ \lambda > 0.\]
	\end{corollary}
	
	\begin{proof}
		One can follow a similar argument to that in Lemma \ref{lem:concentration} and use Proposition \ref{prop:WC} to obtain the above estimate.
	\end{proof}
	
	\subsection{Fourier-Lebesgue space and Bourgain space}
	
	Let $\mathcal{X}$ be a Banach space and $F(t,x)\in \mathcal{S}^\prime(\R\times \mathcal{X})$, denote that
	\begin{equation*}
		\|F\|_{\mathcal{F}\mathcal{L}^{q,\gamma}(\R;\mathcal{X})} := \big\|\langle\tau\rangle^\gamma \|\widehat{F}(\tau,\cdot)\|_{\mathcal{X}}\big\|_{L^q_\tau},
	\end{equation*}
	where $\widehat{F}(\tau)$ is the Fourier transform of $F$ in time variable $t$.
	
	We have the following embedding theorems for Fourier-Lebesgue space
	
	\begin{lemma}[\cite{Burq2024}, Lemma A.5]\label{lem-Pre-embed}
		Let $\mathcal{X}$ be a Banach space, $1<q<\infty$ and $\gamma\in (\frac{1}{q^\prime},1)$. Then, for all $0<\beta\le \gamma-\frac{1}{q^\prime}$ there exists $C>0$ such that for all $F\in C_c^\infty(\R;\mathcal{X})$
		\begin{equation*}
			\|F\|_{C^{0,\beta}(\R;\mathcal{X})} \le C \| F\|_{\mathcal{F}\mathcal{L}^{q,\gamma}(\R;\mathcal{X})}.
		\end{equation*}
		
	\end{lemma}
	
	By performing the frequency decomposition, we can express the function as
	\begin{align*}
		\widehat{F}(\tau,x)=\sum_{n=0}^{\infty}\widehat{\pi_nF}(\tau,x).
	\end{align*}	
	For $p,q,r\in[1,\infty]$, we consider the Fourier-Lebesgue type Bourgain space
	\begin{align*}
		\big\|F\big\|_{X_{p,q,r}^{s,\gamma}}=\big\|\lambda_n^s\langle\tau+\lambda_n^2\rangle^\gamma\widehat{\pi_nF}(\tau,\cdot)\big\|_{\ell_n^pL_\tau^q L_x^r}=\big\|\lambda_n^s\langle\kappa\rangle^\gamma{\big(e^{-it(\Delta_g-1)}F\big)^\wedge}(\kappa,\cdot)\big\|_{\ell_n^pL_\kappa^q L_x^r}.
	\end{align*}
	When $p=q=r=2$, this is the usual $X^{s,b}$ spaces. For fixed $n\in\N$, we denote
	$$ \|F_n\|_{X_{q,r}^{s,\gamma}(E_n)}:=\|\lambda_n^2\cdot e^{it\lambda_n^2}F_n \|_{\mathcal{F}\mathcal{L}^{q,\gamma}(\R;L_x^r(E_n) ) }.
	$$
	We recall the following maximal estimate proved in \cite{Wang-Zhang}, which is a consequence of the Sobolev embedding property:
	\begin{lemma}\label{maximal-easy}
		Let $s,b>\frac{1}{2}$. For any $F\in X^{s,b}$ and $\chi\in C_c^{\infty}(\R)$, we have
		\[
		\|\chi(t)F\|_{L_x^2L_t^{\infty}}\leq C\|F\|_{X^{s,b}}.
		\]
	\end{lemma}
	\begin{remark}\label{rm:maximal-easy}
		Consequently, when $s>\frac{1}{2}$, the almost everywhere convergence property holds for the nonlinear flow \eqref{NLS} with any initial data $\phi \in H^{s}(\mathbb{S}^2)$.
	\end{remark}

	For the linear operator, we can define its Fourier-Lebesgue restriction norms in the following way. Let $\mathcal{T}_n(t):E_n\to E_n$ be a linear operator on $n$-th eigenspace $E_n$ depending on time variable. We denote $T_{n;k,k^\prime}(t)$ by the matrix element of $\mathcal{T}_n$ under the eigenbasis $\{{\bf\mathbf {b}}_{n,k}\}$,
	\begin{align*}
		\mathcal{T}_n(\mathbf{b}_{n,k})=\sum_{|\ell|\leq n}T_{n;\ell,k}(t)\mathbf{b}_{n,\ell}.
	\end{align*}
	The Fourier transform of $\mathcal T_n(t)$ is defined by
	\begin{align*}
		\widehat{\mathcal T}_n(\tau)=\int_{\R}\mathcal{T}_n(t)e^{-it\tau}\,{\rm d}t.
	\end{align*}
	
	We denote $S_n^{q,\gamma}$ and $S_n^{q,\gamma,*}$ by
	\begin{gather*}
		\big\|\mathcal{T}_n\big\|_{S_n^{q,\gamma}}=\big\|\langle\tau+\lambda_n^2\rangle^\gamma\widehat{\mathcal T}_n(\tau)\big\|_{E_n\to L_\tau^qE_n},\\
		\big\|\mathcal T_n\big\|_{S_n^{q,\gamma,*}}=\big\|\langle\tau-\lambda_n^2\rangle^\gamma\widehat{\mathcal T}_n(\tau)\big\|_{E_n\to L_\tau^qE_n},
	\end{gather*}
	where we do not distinguish $E_n$ and $L^2(E_n)$.

	Next, we give small $\sigma\in (0,2^{-100})$ and present $\sigma$ dependent parameters\\ $(q(\sigma),\gamma(\sigma),\gamma_1(\sigma),\delta(\sigma),\theta(\sigma),b(\sigma))$ as follows, which was borrowed from \cite{Burq2024},
	\begin{equation}\label{fml-pre-para}
		\begin{aligned}
			\frac{1}{q(\sigma)} = \sigma,\quad 1-\gamma(\sigma) = \sigma-\sigma^{10},\quad 1-\gamma_1(\sigma) = \sigma-\sigma^{15},\\
			\delta(\sigma) = \sigma^{20},\quad s(\sigma) = \alpha-\frac{1}{2}-100\sigma,\quad \theta(\sigma) = \sigma^5.
		\end{aligned}
	\end{equation}
	Here, parameter $b(\sigma)$ is restricted in the interval $(\frac{1}{2},\frac{1}{2}+\theta(\sigma))$.
	
	Since $\alpha > 1$, we may take $\sigma$ small enough(depend on $\alpha$) such that
	\begin{equation*}
		s(\sigma) = \alpha-\frac{1}{2}-100\sigma > \frac{1}{2} + 100\sigma.
	\end{equation*}
	Observing that $1<\gamma_1(\sigma)q^\prime(\sigma)<\gamma(\sigma)q^\prime(\sigma)$, thus we can find
	\begin{equation*}
		0<\delta(\sigma)\ll \gamma_1(\sigma)-\frac{1}{q^\prime(\sigma)}\ll \gamma(\sigma)-\frac{1}{q^\prime(\sigma)}\ll \theta(\sigma) \ll \frac{1}{q(\sigma)}\ll s(\sigma)-\frac{1}{2}\ll 1.
	\end{equation*}
	The above construction is related to the $\operatorname{Loc}(N)$ condition in Section \ref{sec:Ansatz}.

	\section{Deterministic counterexample}
	In this section, we give the alternative counterexample for $L^r$ maximal estimate for linear Schr\"odinger equation on $\S^2$ with some $r\geq2$.

	\subsection{An arithmetic result}
	
	We recall the following classical result concerning an average order of the Euler indicator function (see Chapter 2,3 of \cite{Tenenbaum})
	$$ \phi(n):=\sum_{\substack{ 1\leq m\leq n\\ 
			\mathrm{gcd}(m,n)=1 } }1. 
	$$ 	
	Recall the Möbius function :
	$$ \mu(n):=
	\begin{cases}
		(-1)^{s},\quad &\text{ if } n=p_1p_2\cdots p_s, \; p_j \text{ are dinstinct primes},\\
		0, \quad &\text{ if } n \text{ is divided by a square of a prime}.
	\end{cases}
	$$
	By the well-known Möbius inversion formula
	$$ \sum_{d|n}\mu(d)=\delta(n),
	$$
	$$ \phi(n)=\sum_{d|n}\mu(d)\frac{n}{d}.
	$$
	Consequently, for any $\lambda\geq 1$,
	\begin{align*}
		\sum_{n\leq \lambda}\frac{\phi(n)}{n}=&\sum_{n\leq \lambda}\sum_{d|n}\frac{\mu(d)}{d}
		=\sum_{d\leq \lambda}\frac{\mu(d)}{d}\left\lfloor 
		\frac{\lambda}{d}
		\right\rfloor\\
		=&\lambda\sum_{d\leq \lambda}\frac{\mu(d)}{d^2}+O(\log\lambda).
	\end{align*}
	Again, using the Möbius inversion formula, we have
	$$ \Big(\sum_{d\geq 1}\frac{\mu(d)}{d^2}\Big)\Big(\sum_{d\geq 1}\frac{1}{d^2}\Big)=1.
	$$
	We deduce that
	\begin{align}\label{eq:ordresmoyennes}
		\sum_{n\leq \lambda}\frac{\phi(n)}{n}=\frac{6}{\pi^2}\lambda+O(\log\lambda),\quad \lambda\rightarrow+\infty.
	\end{align}
	Using this and Chebyshev, we have the following consequence :
	\begin{lemma}\label{lem:ordresmoyennes}
		Let $b>a>0$ be two constants. Then for $c_0\in(0,6/\pi^2)$ and for sufficiently large $\lambda$,
		$$ \#\Big\{ n\in [a\lambda,b\lambda]: \phi(n)\geq \Big(\frac{6}{\pi^2}-c_0\Big)n \Big\}\geq (b-a)\lambda-O_{c_0}(\log\lambda),
		$$
		as $\lambda\rightarrow+\infty$.
	\end{lemma}
	\begin{proof}
		By \eqref{eq:ordresmoyennes},
		$$ \sum_{a\lambda\leq n\leq b\lambda} \frac{\phi(n)}{n}=\frac{6}{\pi^2}(b-a)\lambda +O(\log\lambda).
		$$	
		By Chebyshev, we have
		\begin{align*}
			\#\Big\{n \in[a\lambda,b\lambda]: \frac{6}{\pi^2}-\frac{\phi(n)}{n}\geq  c_0 \Big\}\leq & \frac{1}{c_0}O(\log\lambda).
		\end{align*}
		This completes the proof.
	\end{proof}

	\subsection{Proof of Theorem \ref{thm:maxestcount}}
	Since $\S^2$ can be embedded into $\R^3$, we can write
	\begin{equation*}
		\S^2 := \{ (x_1,x_2,x_3)\in\R^3 : x_1^2 + x_2^2 + x_3^2 =1 \}.
	\end{equation*}
	Thus, the highest weight spherical harmonic is given by $\phi_n(x)=(x_1+ix_2)^n$. Using polar coordinates, we can further rewrite $\phi_n(x)=(\lambda e^{i\theta})^n$,
	with
	\begin{equation*}
		\lambda = \sqrt{x_1^2+x_2^2},\quad \theta = \arctan\left( \frac{x_2}{x_1}\right).
	\end{equation*}
	Here $\lambda\in[0,1]$ and $\theta\in[0,2\pi)$.
	Direct computation shows
	$$	 \|\phi_n\|_{H^s(\mathbb{S}^2)}\sim n^{s-\frac{1}{4}}.
	$$
	Let $\varphi$ be a  bump function $\varphi$ satisfying
	\begin{equation*}
		\supp\varphi \subset \left(\frac{1}{2},\frac{5}{2}\right);\,\,\varphi(x)\equiv 1,\,\,\mbox{if }  x\in \left[\frac{3}{4},\frac{9}{4}\right].
	\end{equation*}
	Let $f_N:= \sum\limits_{n\in\Z} \varphi(\frac{n}{N}) \phi_n(x)$ where $N\geq2$ is a dyadic number. Thanks to the orthogonality, we have 
	\begin{equation*}
		\|f_N\|_{\dot{H}^s(\S^2)} \lesssim \Big( \sum_{n=0}^{3N} \|\phi_n\|^2_{H^s(\S^2)} \Big)^{\frac{1}{2}} \sim \Big( \sum_{n=0}^{3N} n^{2s-\frac{1}{2}} \Big)^{\frac{1}{2}} \sim N^{s+\frac{1}{4}}.
	\end{equation*}
	By the spectral resolution, one can write
	\begin{align*}
		e^{it\Delta_{\S^2}}f_N:=\sum_{n\in\mathbb{Z}}\varphi\big( \frac{n}{N}\big) e^{itn(n+1)}e^{in\theta}\lambda^n.
	\end{align*}
	For convenience, we denote $S_N(t,\theta,\lambda):=e^{it\Delta_{\S^2}}f_N$.
	Also, we define the maximal Schr\"odinger operator by
	\begin{equation*}
		MS_N(\theta,\lambda) = \sup_{t\in[0,1]} \left\vert \sum_{n\in\mathbb{Z}}\varphi\big( \frac{n}{N}\big) e^{itn(n+1)}e^{in\theta}\lambda^n \right\vert.
	\end{equation*}

	By direct computation, one can verify that
	\begin{align}\label{polar}
		\Big\|\sup_{t\in [0,1]}|e^{it\Delta_{\S^2}}f_N|\Big\|_{L^r(\S^2)}^r = 2\int_0^1 \int_0^{2\pi} \frac{\lambda}{\sqrt{1-\lambda^2}} \vert MS_N(\theta,\lambda)\vert^r \mathrm{d}\theta \mathrm{d}\lambda.
	\end{align}
	Then, it is sufficient to give a pointwise lower bound for $MS_N(\theta,\lambda)$.
	
	We assert that there exists $C>0$ independent of $\theta,\lambda, N$ and a measurable set $\mathcal{E}$ (defined in \eqref{fml-def-setE}) with $\operatorname{Mes}(\mathcal{E})\sim 1$ (where $\operatorname{Mes}$ denotes Lebesgue measure on $[0,2\pi]$), such that
	\begin{equation}\label{fml-asser-MSN-lower}
		\left| MS_N(\theta,\lambda)\right| \ge C N^{\frac{3}{4}}\lambda^{N}, \quad \text{if}\,\, (\theta,\lambda) \in \mathcal{E}\times(1-N^{-1-},1).
	\end{equation}
	With this claim, we can prove the lower bound for maximal estimate. Using \eqref{polar} and \eqref{fml-asser-MSN-lower}, we have
	\begin{equation*}
		\begin{aligned}
			\Big\|\sup_{t\in [0,1]}|e^{it\Delta_{\S^2}}f_N|\Big\|_{L^r(\S^2)}^r &\gtrsim N^{\frac{3r}{4}} \int_{1-N^{-1-}}^1 \int_0^{2\pi} \frac{\lambda^{Nr+1}}{\sqrt{1-\lambda^2}} \mathbf{1}_{\mathcal{E}}(\theta) \mathrm{d}\theta \mathrm{d}\lambda\\
			&\gtrsim N^{\frac{3r}{4}-\frac{1}{2}-}.
		\end{aligned}
	\end{equation*}
	It implies
	\begin{equation*}
		\begin{aligned}
			\Big\|\sup_{t\in [0,1]}|e^{it\Delta_{\S^2}}f_N|\Big\|_{L^r(\S^2)}
			\gtrsim N^{\frac{3}{4}-\frac{1}{2r}-}.
		\end{aligned}
	\end{equation*}
	
	Since $\|f_N\|_{H^s} \lesssim N^{s+\frac{1}{4}}$, thus for $s<\frac{1}{2}-\frac{1}{2r}$
	\begin{equation*}
		\Big\| \sup_{t\in[0,1]} \left| e^{it\Delta}f_N \right| \Big\|_{L^r_x} > C\|f_N\|_{\dot{H}^s}.
	\end{equation*}
	Therefore, we conclude the proof of Theorem \ref{thm:maxestcount}.
	
	Now, we turn to prove  \eqref{fml-asser-MSN-lower}. 
	We take $C_0^{-1}N^{\frac{1}{2}}<p<C_0N^{\frac{1}{2}}$ for a large integer $C_0$ to be fixed later. We also postponed the argument on $\mathcal{E}$ later.

	Next, for an odd integer $p$ and an even integer $q$ such that $\mathrm{gcd}(p,q)=1$, by taking $t = \frac{2\pi}{p}$ and $\theta = \frac{2\pi q}{p}$, one obtains
	\begin{equation*}
		\begin{aligned}
			S_N\left(\frac{2\pi}{p},\frac{2\pi q}{p},\lambda\right) &= \sum_{n\in\mathbb{Z}}\varphi\left( \frac{n}{N}\right) \lambda^n e^{2\pi i\frac{ n(n+1)}{p}} e^{2\pi i\frac{nq}{p}}\\
			&=\sum_{k=1}^{p} \sum_{m\in\mathbb{Z}} \varphi\left( \frac{pm+k}{N}\right)\lambda^{pm+k} e^{2\pi i\frac{ (pm+k)(pm+k+1)}{p}} e^{2\pi i(pm+k)\frac{q}{p}}\\
			&= \sum_{k=1}^p e^{2\pi i(\frac{k^2+k(q+1)}{p})} \sum_{m\in\mathbb{Z}}\varphi\left( \frac{pm+k}{N}\right) \lambda^{pm+k}.
		\end{aligned}
	\end{equation*}
	For $x\in\R$, we define
	\begin{equation*}
		F(x) = G\left(\frac{px+k}{N}\right) = \varphi(\frac{px+k}{N})\lambda^{px+k}.
	\end{equation*}
	Thus, by introducing the change of variable $y = \frac{px+k}{N}$ we obtain
	\begin{equation*}
		\begin{aligned}
			\widehat{F}(\xi) &= \int_{\R} \varphi(\frac{px+k}{N})\lambda^{px+k} e^{-2\pi ix\xi} \mathrm{d} x =e^{2\pi i\frac{k}{p}\xi}\frac{N}{p} \widehat{G}\left(\frac{N\xi}{p}\right).
		\end{aligned}
	\end{equation*}
	
	By Poisson's summation, we have
	\begin{equation*}
		\sum_{m\in\mathbb{Z}}\varphi\left( \frac{pm+k}{N}\right) \lambda^{pm+k} = \sum_{\xi\in\Z}e^{2\pi i\frac{k}{p}\xi}\frac{N}{p} \widehat{G}\left(\frac{N\xi}{p}\right) = \frac{N}{p}\widehat{G}(0) + \sum_{\xi\in\Z\backslash\{0\}}e^{2\pi i\frac{k}{p}\xi}\frac{N}{p} \widehat{G}\left(\frac{N\xi}{p}\right),
	\end{equation*}
	which allows us split $S_N(2\pi/p,2\pi q/p,\lambda)$ into two parts:
	\begin{equation*}
		S_N\left(\frac{2\pi}{p},\frac{2\pi q}{p},\lambda\right) = \underbrace{\frac{N}{p}\sum_{k=1}^p e^{2\pi i(\frac{k^2+k(q+1)}{p})} \widehat{G}(0)}_{\mathcal{A}}  + \underbrace{\frac{N}{p}\sum_{k=1}^p e^{2\pi i(\frac{k^2+k(q+2)}{p})} \sum_{\xi\in\mathbb{Z}\backslash\{0\}}\widehat{G}\left( \frac{N\xi}{p} \right)}_{\mathcal{B}}.
	\end{equation*}
	
	\textbf{\underline{Bound on the main part $\mathcal{A}$} :} 
	Direct computation yields that
	\begin{equation*}
		\frac{N}{p}\widehat{G}(0) \ge \frac{N}{p} \int_{\frac{3}{4}}^1 \lambda^{Nx} \mathrm{d} x \ge \frac{N}{p} \lambda^N.
	\end{equation*}
	Next, by the Gauss sum \cite{Hardy}, we get
	\begin{align}\label{fml-Gaussian}
		\sum_{n=0}^{p-1} e^{2\pi i\frac{n^2 + (q+1)n}{p}} = \omega_p \sqrt{p}e^{2\pi i\frac{m(q+1)^2}{p}},
	\end{align}
	where $4m\equiv 1(\operatorname{mod} p)$ and
	\begin{align*}
		\omega_p =
		\begin{cases}
			1,\quad &p\equiv 1 (\operatorname{mod} 4),\\
			-i,\quad &p\equiv 3 (\operatorname{mod} 4).
		\end{cases}
	\end{align*}
	For $C_0^{-1}N^{\frac{1}{2}}<p<C_0N^{\frac{1}{2}}$ with a positive integer $C_0>0$, we have for sufficiently large $N$, 
	\begin{equation}\label{fml-G0}
		|\mathcal{A}| \ge \sqrt{p}\cdot \lambda^N\frac{N}{p} \geq C_0^{-\frac{1}{2}}N^{\frac{3}{4}}\lambda^N.
	\end{equation}
	
	\textbf{\underline{Bound on the remainder $\mathcal{B}$} : } 
	Applying integrating by parts, for all $k\in \N$, we get
	\begin{equation*}
		\left|\widehat{G}\left( \frac{N\xi}{p} \right)\right| = \left( \frac{p}{N\xi} \right)^k \int_{\frac{1}{2}}^{\frac{5}{2}}\left| \frac{\mathrm{d}^k}{\mathrm{d}y^k}\left( \varphi(y)\lambda^{Ny} \right) \right| \mathrm{d} y.
	\end{equation*}
	Since $\varphi(y)$ is smooth enough, we only need consider the case that all derivatives fall on $\lambda^{Ny}$. Recall that $\lambda\in(1-N^{-1-},1)$, we have a pointwise bound:
	\begin{equation}\label{fml-lambdaN-pointwise}
		\left| \frac{\mathrm{d}^k}{\mathrm{d}y^k}(\lambda^{Ny}) \right| = \lambda^{Ny} |N\log(\lambda)|^k \le \lambda^{Ny}.
	\end{equation}
	Recall that $\lambda\in(1-N^{-1-},1)$, we can find for all $k\in\N$ it follows that $(N/p)^k\lambda^{N/2}\le \lambda^N$. With this in hand, we take $k=2$ and apply \eqref{fml-Gaussian} to get that
	\begin{equation}\label{fml-Gxi-ne0}
		\begin{aligned}
			\left| \mathcal{B} \right| \le \sqrt{p}\cdot \left| \sum_{\xi\in\Z\backslash\{0\}} \frac{N}{p} \widehat{G}\left(\frac{N\xi}{p}\right)\right| \le 4 \sum_{\xi\in\Z\backslash\{0\}} |\xi|^{-2} \frac{p^{\frac{3}{2}}}{N} \ll N^{\frac{3}{4}} \lambda^N.
		\end{aligned}
	\end{equation}
	From \eqref{fml-lambdaN-pointwise}, the above estimate is independent of $k$ and $N$.
	
	Combine \eqref{fml-G0} with \eqref{fml-Gxi-ne0}, we get
	\begin{equation}\label{fml-SN-bound}
		\left|S_N\left(\frac{2\pi}{p},\frac{2\pi q}{p},\lambda\right)\right| \ge |\mathcal{A}| - |\mathcal{B}| \ge \frac{1}{2} C_0^{-\frac{1}{2}}N^{\frac{3}{4}} \lambda^N.
	\end{equation}
	\medskip

	\textbf{\underline{Construction of $\mathcal{E}$} : }

	Then, we turn to construct the set $\mathcal{E}$, following the strategy in \cite{MV}. For odd integer $p$ and even integer $q$, we define 
	\begin{equation*}
		E(N,p,q) := \left( \frac{2\pi q}{p}-\frac{1}{10C_0^2\pi N},\frac{2\pi q}{p}+\frac{1}{10C_0^2\pi N} \right). 
	\end{equation*}
	Next, we also define
	\begin{equation}\label{fml-def-setE}
		\mathcal{E} = \bigcup_{\substack{p \text{ odd}\\ C_0^{-1}N^{1/2}<p<C_0N^{1/2}}} \bigcup_{\substack{q : \mathrm{gcd}(p,q)=1\\ 2\pi q/p \in (0,\pi) }} E(N,p,q).
	\end{equation}
	We observe that the intervals $E(N,p,q)$ are disjoint. Indeed, for distinct pairs  $(p_1,q_1),(p_2,q_2)$, since $\gcd(p_j,q_j)=1,j=1,2$, we have such that $\frac{q_1}{p_1}\neq\frac{q_2}{p_2}$. Thus
	$$ \Big|\frac{q_1}{p_1}-\frac{q_2}{p_2}\Big|=\frac{|p_1q_2-p_2q_1|}{p_1p_2}\geq \frac{1}{p_1p_2}\geq \frac{1}{C_0^2N}.
	$$
	It follows immediately that $E(N,p_1,q_1)\cap E(N,p_2,q_2)=\emptyset$.
	
	Since for fixed $p\in (C_0^{-1}\sqrt{N},C_0\sqrt{N} )$, the number of intervals $E(N,p,q)$ defining $\mathcal{E}$ is at most $\phi(p)$, where $\phi(p)$ is the Euler indicator function, by Lemma \ref{lem:ordresmoyennes}, we deduce that there are at least $cN$ such intervals. Consequently, $\mathrm{Mes}(\mathcal{E})\sim 1$.

	\medskip
	
	To complete the proof,	for $\widetilde{\theta}$ satisfying $|\frac{2\pi q}{p} - \widetilde{\theta}| \le \frac{1}{10C_0^2\pi N}$, the mean-value theorem yields
	\begin{equation}\label{fml-SN-diff}
		\begin{aligned}
			&\left| S_{N}\left( \frac{2\pi}{p},\widetilde{\theta},\lambda \right) - S_{N}\left( \frac{2\pi}{p},\frac{2\pi q}{p},\lambda \right) \right|\\
			=&\Big| \sum_{n\in\Z}\varphi\left(\frac{n}{N}\right) \lambda^n e^{2\pi i\frac{n(n+1)}{p}} \left( e^{2\pi i\frac{nq}{p}} - e^{in\widetilde{\theta}} \right)\Big|\\
			\leq&\Big| \sum_{k=1}^p e^{2\pi i(\frac{k^2+k(q+1)}{p})}\left( 1 - e^{ik(\widetilde{\theta}-(2\pi q)/p)} \right) \sum_{|m|\sim N/p}\varphi\left( \frac{pm+k}{N}\right) \lambda^{pm+k} \Big|\\
			& +\Big| \sum_{k=1}^p e^{2\pi i(\frac{k^2+k}{p})} e^{ik\widetilde{\theta}} \sum_{|m|\sim N/p}\varphi\left( \frac{pm+k}{N}\right) \left( 1 - e^{imp\tilde{\theta}} \right) \lambda^{pm+k}\Big|\\
			\le& \frac{p^2}{N} \cdot \frac{N}{p} \lambda^{\frac{N}{2}} + \frac{1}{10C_0^2}\cdot \sqrt{p}\frac{N}{p}\lambda^{\frac{N}{2}}\\
			\le& \frac{1}{4} C_0^{-\frac{3}{2}} N^{\frac{3}{4}} \lambda^N.
		\end{aligned}
	\end{equation}
	In the last inequality, we have used the fact that $\lambda\in (1-N^{-1-},1)$ and $N$ sufficiently large such that $\lambda^{N/2} \le 3 \lambda^{N}$.
	
	Consequently, combining \eqref{fml-SN-bound} with \eqref{fml-SN-diff}, and take $C_0$ sufficiently large, we have proved \eqref{fml-asser-MSN-lower} which accomplishes the proof of Theorem \ref{thm:maxestcount}.
	
	\section{Linear Schr\"odinger equation}\label{section-LSE}
	
	In this section, our main goal is to prove Theorem \ref{thm:2}, the almost sure convergence for the linear flow associated with randomized initial data. 
	
	\begin{proposition}\label{prop:convergence of linear solution}
		Let $\phi_\alpha^\omega$ be as in \eqref{randomize initial data}. Then, for any $\varepsilon>0$,
		\begin{equation}\label{convergence of linear solution}
			\P\bigg(\bigg\{\omega\in\Omega \,\bigg|\, \lim_{t\to0}\big\|e^{it(\Delta_g-1)}\phi_\alpha^\omega-\phi_\alpha^\omega\big\|_{L_x^\infty(\S^2)}=0\bigg\}\bigg)\ge 1-\varepsilon.
		\end{equation}
	\end{proposition}
	
	To handle the linear and nonlinear problems within a common notation, we only prove the convergence for $e^{it(\Delta_g-1)}\phi_\alpha^\omega$. Setting $v=e^{it}u$, where $u=e^{it(\Delta_g-1)}\phi_\alpha^\omega$, immediately yields the convergence for $e^{it\Delta_g}$.
	
	Let $\mathbf{b}_{n,k}(x)$ be an eigenfunction in the $n$-th eigenspace, and let $f(x)$ on $\S^2$ be given by
	\begin{equation}\label{frequency}
		f(x) = \sum_{n=0}^{\infty} \sum_{|k|\le n} c_{n,k} \mathbf{b}_{n,k}(x),
	\end{equation}
	we denote by $P_N$ the frequency projection onto the annulus of size $N\in 2^{\N}$, i.e.
	\begin{equation}\label{frequency projection}
		P_N(f)(x) = \sum_{\frac{N}2 < \vert \lambda_n \vert \le N} \sum_{|k|\le n} c_{n,k} \mathbf{b}_{n,k}(x).
	\end{equation}

	To prove Proposition \ref{prop:convergence of linear solution}, we first present a probabilistic $L^p$ estimate and an improved Strichartz estimate:
	\begin{proposition}\label{prop:large deviation estimate}
		For $ \phi_{\alpha}^{\omega} $ as in \eqref{randomize initial data}, there exists a constant $ C > 0 $ such that
		\begin{equation}\label{large deviation estimate}
			\big\Vert P_N \big( \phi_{\alpha}^{\omega} \big) \big\Vert_{L_{\omega}^p(\Omega)} \le C \sqrt{p} N^{-\alpha + 1}, \qquad \text{\rm uniformly in} ~ x \in \S^2.
		\end{equation}
	\end{proposition}
	
	\begin{proof}
		By the definition of $ P_N $ in \eqref{frequency projection}, one has
		\begin{equation*}
			P_N \big( \phi_{\alpha}^{\omega} \big) = \sum_{\frac{N}2 < \lambda_n \le N} \frac1{\lambda_n^{\alpha - \frac12}} \frac1{\sqrt{2n + 1}} \sum_{\vert k \vert \le n} g_{n, k}(\omega) \mathbf{ b}_{n, k}(x).
		\end{equation*}
		Hence, one can make the use of Khinchin's inequality for
		\begin{align}\label{def-cnk}
			c_{n, k} = \frac1{\lambda_n^{\alpha - \frac12}} \frac1{\sqrt{2n + 1}} \mathbf{ b}_{n, k}(x)
		\end{align}
		to obtain
		\begin{equation*}
			\text{LHS of} ~ \eqref{large deviation estimate} \le C \sqrt{p} \bigg( \sum_{\frac{N}2 < \lambda_n \le N} \frac1{\lambda_n^{2\alpha - 1}} \frac1{2n + 1} \sum_{\vert k \vert \le n} \vert \mathbf{ b}_{n, k}(x) \vert^2 \bigg)^{\frac12}.
		\end{equation*}
		Then, by the exact Weyl's law \eqref{Local-Weyl}, we obtain
		\begin{equation*}
			\big\Vert P_N \big( \phi_{\alpha}^{\omega} \big) \big\Vert_{L_{\omega}^p(\Omega)} \le C \sqrt{p} N^{-\alpha + 1}.
		\end{equation*}
	\end{proof}
	
	\begin{remark} \label{remark:frequency of linear solution}
		Similar as Proposition \ref{prop:large deviation estimate}, one can use $ \tilde{c}_{n,k} = \frac{e^{-it \lambda_n^2}}{\lambda_n^{\alpha-\frac12}} c_{n,k} $ instead of $ c_{n,k} $ in \eqref{def-cnk} to obtain
		\begin{equation}\label{frequency of linear solution}
			\big\Vert P_N \big( e^{it(\Delta_g-1)} \phi_\alpha^\omega \big) \big\Vert_{L_{\omega}^p(\Omega)} \le C \sqrt{p} N^{-\alpha +1}, \qquad \text{\rm uniformly in} ~ (t, x) \in \R \times \S^2.
		\end{equation}
	\end{remark}
	
	Next, we show that 
	
	\begin{corollary}[Improved $ L^p $ estimate]\label{coro:probability of linear solution}
		For any $ \varepsilon > 0 $, we have 
		\begin{equation}\label{probability of linear solution}
			\P \bigg( \bigg\{ \omega \in \Omega ~ \bigg\vert ~ \big\Vert P_N \big( e^{it(\Delta_g-1)} \phi_\alpha^\omega \big) \big\Vert_{L_x^p(\S^2)} \lesssim N^{-\alpha+1} \ln \big( \frac1{\sqrt{\varepsilon}} \big) \bigg\} \bigg) \ge 1 - \varepsilon  
		\end{equation}
		holds for any $ 2 \le p < \infty $, $ N \in 2^{\N} $ and
		\begin{equation}\label{probability of linear solution1}
			\P \bigg( \bigg\{ \omega \in \Omega ~ \bigg\vert ~  \big\Vert P_N \big( e^{it(\Delta_g-1)} \phi_\alpha^\omega \big) \big\Vert_{L_x^{\infty}(\S^2)} \lesssim N^{-\alpha+1 +} \ln \big( \frac1{\sqrt{\varepsilon}} \big) \bigg\} \bigg) \ge 1 - \varepsilon
		\end{equation}
		holds for any $ N \in 2^{\N} $.
	\end{corollary}
	
	\begin{proof}
		Fix $ N \in 2^{\N} $, $ t \in \R $ and $ p \in [2, \infty) $.
		Then, for any $ q \in [p, \infty) $, by Minkowski's inequality and \eqref{frequency of linear solution}, one has
		\begin{align*}
			\bigg( \int_{\Omega} \big\Vert P_N \big( e^{it(\Delta_g-1)} \phi_\alpha^\omega \big) \big\Vert_{L_x^p(\S^2)}^q {\rm d}\P(\omega) \bigg)^{\frac1q} 
			\le & ~ \bigg\Vert \big\Vert P_N \big( e^{it(\Delta_g-1)} \phi_\alpha^\omega \big) \big\Vert_{L_{\omega}^q(\Omega)} \bigg\Vert_{L_x^p(\S^2)} \\
			\le & ~ C \sqrt{q} N^{-\alpha+1},
		\end{align*}
		which means that
		\begin{equation}\label{expectation}
			\E_{\omega} \bigg[ \big\Vert P_N \big( e^{it(\Delta_g-1)} \phi_\alpha^\omega \big) \big\Vert_{L_x^p(\S^2)}^q \bigg] \le C^q q^{\frac{q}2} N^{q(-\alpha+1)}.
		\end{equation}
		Furthermore, by Markov's inequality and \eqref{expectation}, we get for any $q \in [p, \infty)$
		\begin{align} 
			\P \bigg( \bigg\{ \omega \in \Omega ~ \bigg\vert ~ \big\Vert P_N \big( e^{it(\Delta_g-1)} \phi_\alpha^\omega \big) \big\Vert_{L_x^p(\S^2)} > \lambda \bigg\} \bigg)
			\le & ~ \lambda^{-q} \E_{\omega} \bigg[ \big\Vert P_N \big( e^{it(\Delta_g-1)} \phi_\alpha^\omega \big) \big\Vert_{L_x^p(\S^2)}^q \bigg] \nonumber \\ 
			\le & ~ C^q \lambda^{-q} q^{\frac{q}2} N^{q(-\alpha+1)}. \label{probability of tail}
		\end{align}
		
		{\bf We claim} that there exists $ \tilde{C} := \tilde{C}(N, \lambda, p) > 0 $ and a constant $ c > 0 $ such that
		\[ 
		\text{LHS of} ~ \eqref{probability of tail} \le \tilde{C} e^{-c N^{2(\alpha-1)} \lambda^2}, \qquad \forall ~ p \ge 2. 
		\]
		Indeed, constructing the auxiliary function $ f(q) := \lambda^{-q} q^{\frac{q}2} N^{q(-\alpha+1)} $ on $ [p, \infty) $, one can find that its minimum value is $ f(\frac{N^{2(\alpha-1)} \lambda^2}e) = e^{-\frac{N^{2(\alpha-1)} \lambda^2}{2e}} $ if $ \frac{N^{2(\alpha-1)} \lambda^2}e \in [p, \infty) $.
		Now, one may obtain that
		\[ 
		\text{RHS of} ~ \eqref{probability of tail} \lesssim e^{-\frac{N^{2(\alpha-1)} \lambda^2}{2e}}. 
		\]
		Choose $ \tilde{C} = C^{\frac{N^{2(\alpha-1)} \lambda^2}e} $ and $ c = \frac1{2e} $, then claim will hold for $ p \in \big[ 2, \frac{N^{2(\alpha-1)} \lambda^2}e \big] $ immediately.
		On the other hand, for $ p \in (\frac{N^{2(\alpha-1)} \lambda^2}e, \infty) $, the trivial bound 
		\[ 
		\text{LHS of} ~ \eqref{probability of tail} \le 1 = e^p e^{-p} \le \tilde{C} e^{-c N^{2(\alpha-1)} \lambda^2} 
		\] 
		means that claim will hold for $ \tilde{C} = e^p $ and $ c = \frac1e $.
		
		Finally, the claim above implies that \eqref{probability of linear solution} for $ \varepsilon = e^{-c N^{2(\alpha-1)} \lambda^2} $.
		Furthermore, \eqref{probability of linear solution1} can be obtained by \eqref{probability of linear solution} and Bernstein's inequality, which completes the proof.
	\end{proof}
	
	\begin{proposition}[Improved Strichartz estimate]\label{prop:improved Strichartz estimate}
		For any $ \varepsilon > 0 $, we have  
		\begin{equation}\label{improved Strichartz estimate}
			\P \bigg( \bigg\{ \omega \in \Omega ~ \bigg\vert ~  \big\Vert e^{it(\Delta_g-1)} P_N \phi_\alpha^\omega \big\Vert_{L_{t,x}^p(\S^2)} \lesssim N^{-\alpha+1} \ln \big( \frac1{\sqrt{\varepsilon}} \big) \bigg\} \bigg) \ge 1 - \varepsilon 
		\end{equation}
		holds for any $ 2 \le p < \infty $, $ N \in 2^{\N} $ and
		\begin{equation}\label{improved Strichartz estimate1}
			\P \bigg( \bigg\{ \omega \in \Omega ~ \bigg\vert ~  \big\Vert e^{it(\Delta_g-1)} P_N \phi_\alpha^\omega \big\Vert_{L_{t,x}^{\infty}(\S^2)} \lesssim N^{-\alpha+1 +} \ln \big( \frac1{\sqrt{\varepsilon}} \big) \bigg\} \bigg) \ge 1 - \varepsilon
		\end{equation}
		holds for any $ N \in 2^{\N} $.
	\end{proposition}
	
	\begin{proof}
		The proof is similar to that of Corollary \ref{coro:probability of linear solution}.
	\end{proof}

	\begin{proof}[Proof of Proposition \ref{prop:convergence of linear solution}]
		Choose $\eta>0$ such that $\eta<\alpha-1$. For each dyadic $N\ge 2$, define the event
		\[
		E_N:=\left\{\omega\in\Omega:\,
		\|P_N\phi_\alpha^\omega\|_{L_x^\infty(\S^2)}
		+\|e^{it(\Delta_g-1)}P_N\phi_\alpha^\omega\|_{L_{t,x}^\infty(\R\times\S^2)}
		\le C N^{-\alpha+1+\eta}\log N\right\}.
		\]
		By \eqref{probability of linear solution1} and \eqref{improved Strichartz estimate1}, applied with $\varepsilon=N^{-4}$, we may choose $C>0$ so that
		\[
		\P(E_N^c)\lesssim N^{-4}.
		\]
		Since $\sum_{N\in 2^{\mathbb{N}}}\P(E_N^c)<\infty$, the Borel-Cantelli lemma yields an event $\Omega_{\lin}\subset\Omega$ of full probability such that, for every $\omega\in\Omega_{\lin}$, there exists $N_0(\omega)$ with $\omega\in E_N$ for all dyadic $N\ge N_0(\omega)$.
		
		Fix $\omega\in\Omega_{\lin}$. Since $\eta<\alpha-1$, the dyadic series
		\[
		\sum_{N\in 2^{\mathbb{N}}} N^{-\alpha+1+\eta}\log N
		\]
		converges. Hence, for every dyadic $M\ge N_0(\omega)$,
		\begin{align*}
			\sup_{t\in\R}\left\|\sum_{N>M}e^{it(\Delta_g-1)}P_N\phi_\alpha^\omega\right\|_{L_x^\infty}
			&\le \sum_{N>M}\|e^{it(\Delta_g-1)}P_N\phi_\alpha^\omega\|_{L_{t,x}^\infty}\\
			&\lesssim \sum_{N>M}N^{-\alpha+1+\eta}\log N,
		\end{align*}
		and similarly,
		\[
		\left\|\sum_{N>M}P_N\phi_\alpha^\omega\right\|_{L_x^\infty}
		\lesssim \sum_{N>M}N^{-\alpha+1+\eta}\log N.
		\]
		Therefore, both tails converge to $0$ uniformly as $M\to\infty$.
		
		For a fixed dyadic $M$, the low-frequency part $P_{\le M}\phi_\alpha^\omega$ is smooth. Hence
		\[
		\lim_{t\to0}\|e^{it(\Delta_g-1)}P_{\le M}\phi_\alpha^\omega-P_{\le M}\phi_\alpha^\omega\|_{L_x^\infty(\S^2)}=0.
		\]
		Combining the dyadic decomposition
		\[
		e^{it(\Delta_g-1)}\phi_\alpha^\omega-\phi_\alpha^\omega
		=\sum_{N>M}e^{it(\Delta_g-1)}P_N\phi_\alpha^\omega
		+\big(e^{it(\Delta_g-1)}P_{\le M}\phi_\alpha^\omega-P_{\le M}\phi_\alpha^\omega\big)
		-\sum_{N>M}P_N\phi_\alpha^\omega,
		\]
		we conclude that
		\[
		\lim_{t\to0}\|e^{it(\Delta_g-1)}\phi_\alpha^\omega-\phi_\alpha^\omega\|_{L_x^\infty(\S^2)}=0
		\]
		for every $\omega\in\Omega_{\lin}$. Since $\P(\Omega_{\lin})=1$, this proves the proposition.
	\end{proof}
	In summary, the proof of Theorem \ref{thm:2} is now complete.

	\section{Resolution ansatz for the cubic NLS}\label{sec:Ansatz}

	In this section, we recall from \cite{Burq2024} only the ingredients of the random averaging operator (RAO) ansatz that will be used in Section \ref{sec:nonlinear-pointwise}. In particular, we keep the dyadic decomposition of the truncated flow, the time-localized objects $\psi_N^\dagger$ and $w_N^\dagger$, the conditional law-equivalence lemma, and the large-probability bounds coming from the $\operatorname{Loc}(N)$ scheme.
	
	\subsection{Dyadic decomposition of the truncated flow}
	
	For each dyadic number $N\in 2^{\mathbb{N}}$, let $u_N$ be the smooth solution to the truncated Wick-ordered cubic NLS
	\begin{equation}\label{fml-WickOrd-NLS}
		\begin{cases}
			i\partial_t u_N + (\Delta_g - 1)u_N = \mathcal{N}(u_N),\\
			u_N(0) = P_{\le N}\phi_\alpha^\omega.
		\end{cases}
	\end{equation}
	Following \cite[Section 1.4]{Burq2024}, one introduces the dyadic increments
	\[
	v_N:=u_N-u_{N/2},\qquad N\ge 2,
	\]
	and decomposes
	\[
	v_N=\psi_N+w_N.
	\]
	Consequently,
	\[
	u_N=\sum_{L\le N}(\psi_L+w_L).
	\]
	The singular part $\psi_N$ is encoded by the random averaging operators $\mathcal{H}^N_n(t):E_n\to E_n$, defined for $N/2<n\le N$, through the identity
	\begin{equation}\label{eq:psi-shell-raw}
		\pi_n\psi_N(t)=\lambda_n^{-\left(\alpha-\frac12\right)}\mathcal{H}^N_n(t)(e_n^\omega).
	\end{equation}
	We shall not use the explicit Duhamel formula for $\mathcal{H}^N_n(t)$.
	
	The main probabilistic well-posedness result from \cite{Burq2024} is the following local convergence theorem.
	
	\begin{theorem}[{\cite[Theorem 1.3]{Burq2024}}]\label{thm:BCST}
		Let $\alpha>1$. There exist constants $C_1>c_1>0$ and $\delta_0>0$ such that, for every $R\ge 1$, there is a measurable set $\Sigma_R\subset H^{\alpha-1-}(\S^2)$ satisfying
		\[
		\mu_\alpha(\Sigma_R^c)<C_1e^{-c_1R^{\delta_0}},
		\]
		with the following property. If $\phi\in\Sigma_R$, then the sequence $(u_N)_N$ converges in
		\[
		C([ -\tau_R,\tau_R ];H^{\alpha-1-}(\S^2)),\qquad \tau_R:=R^{-C_1},
		\]
		to a solution $u$ of \eqref{Wick ordered}. Moreover,
		\[
		u=\psi_{\le \infty}+w_{\le \infty},
		\]
		where
		\[
		\psi_{\le \infty}\in C([ -\tau_R,\tau_R ];\mathcal{C}^{(\alpha-1)-}(\S^2))
		\quad\text{and}\quad
		w_{\le \infty}\in C([ -\tau_R,\tau_R ];H^{s_0}(\S^2))
		\]
		for some $s_0>\frac12$. Here
		\[
		\mathcal{C}^{(\alpha-1)-}(\S^2):=\bigcap_{\beta<\alpha-1} C^{\beta}(\S^2).
		\]
	\end{theorem}
	
	\subsection{Time-localized ansatz and the large-probability event}\label{subsec:time-localized-ansatz}
	
	Let $\chi\in C_c^\infty(\mathbb{R})$ satisfy $\chi\equiv 1$ on $[-\frac12,\frac12]$ and $\operatorname{supp}\chi\subset (-1,1)$. For $0<T<\frac12$, write $\chi_T(t):=\chi(T^{-1}t)$. The time-localized version of the RAO ansatz constructed in \cite[Section 5]{Burq2024} provides, for every dyadic $N$, compactly supported functions $u_N^\dagger$, $\psi_N^\dagger$, $w_N^\dagger$ and operators $\mathcal{H}^{N,\dagger}_n(t):E_n\to E_n$ such that
	\begin{equation}\label{eq:dagger-equals-original}
		u_N^\dagger(t)=u_N(t),\qquad \psi_N^\dagger(t)=\psi_N(t),\qquad w_N^\dagger(t)=w_N(t)
	\end{equation}
	for all $|t|\le T/2$. In addition, for $N/2<n\le N$,
	\begin{equation}\label{eq:psi-dagger-shell}
		\pi_n\psi_N^\dagger(t)=\chi_T(t)\lambda_n^{-\left(\alpha-\frac12\right)}\mathcal{H}^{N,\dagger}_n(t)(e_n^\omega).
	\end{equation}
	Since each truncated solution $u_N$ is smooth, all these objects are defined pathwise for every $\omega\in\Omega$.
	
	We denote by $\mathcal{B}_{\le N}$ the sigma-algebra generated by the Gaussian variables
	\[
	\big\{g_{m,k}^\omega: 0\le m\le N,\ |k|\le m\big\}.
	\]
	For $|t|\le T$, the operator $\mathcal{H}^{N,\dagger}_n(t)$ is unitary on $E_n$ and $\mathcal{B}_{\le N/2}$-measurable; see \cite[Lemma 7.1]{Burq2024}. The following conditional law-equivalence lemma is the key probabilistic tool in Section \ref{sec:nonlinear-pointwise}.
	
	\begin{lemma}[Conditional law-equivalence, {\cite[Lemma 7.1]{Burq2024}}]\label{lem-invariance}
		Let $\ell\in\mathbb{N}^*$ and let $n_1,\dots,n_{\ell}\in (N/2,N]$ be integers, not necessarily distinct. Let $F$ be a bounded Borel measurable functional on $\mathbb{C}\times E_{n_1}\times\cdots\times E_{n_{\ell}}$, and let $Y$ be a $\mathcal{B}_{\le N/2}$-measurable random variable. Then, for every $|t|\le T$,
		\[
		\E\Big[F\big(Y,\mathcal{H}^{N,\dagger}_{n_1}(t)(e_{n_1}^\omega),\dots,\mathcal{H}^{N,\dagger}_{n_{\ell}}(t)(e_{n_{\ell}}^\omega)\big)\,\big|\,\mathcal{B}_{\le N/2}\Big]
		=
		\E\Big[F\big(Y,e_{n_1}^\omega,\dots,e_{n_{\ell}}^\omega\big)\,\big|\,\mathcal{B}_{\le N/2}\Big]
		\]
		almost surely in $\omega$.
	\end{lemma}
	
	The other ingredient needed in Section \ref{sec:nonlinear-pointwise} is the following large-probability statement of the $\operatorname{Loc}(N)$ property. We only record the bounds that will actually be used later.
	
	\begin{proposition}[{\cite[Definition 5.2, Corollary 5.7 and Lemma 7.2]{Burq2024}}]\label{prop:loc-bounds}
		Assume that $\alpha>1$. There exist parameters
		\[
		(q,\gamma,\gamma_1,\delta,s,\theta,b)=(q(\sigma),\gamma(\sigma),\gamma_1(\sigma),\delta(\sigma),s(\sigma),\theta(\sigma),b(\sigma))
		\]
		as in \eqref{fml-pre-para}, with $s>\frac12$ and $b>\frac12$, and constants $R_0\ge 1$, $C_2>c_2>0$, $\delta_1>0$ such that the following holds. For every $R\ge R_0$, set
		\[
		T:=R^{-10/\theta}.
		\]
		Then there exists a good data set $\Sigma_R^{\mathrm{loc}}\subset H^{(\alpha-1)-}(\mathbb{S}^2)$ satisfying
		\[
		\mu_{\alpha}\big((\Sigma_R^{\mathrm{loc}})^c\big)\le C_2e^{-c_2R^{\delta_1}},
		\]
		with the property that, for every initial data $\phi\in\Sigma_R^{\mathrm{loc}}$, every dyadic $N$, and every integer $n$ with $N/2<n\le N$, the time-localized colored term satisfies
		\begin{equation}\label{eq:psin-X0gamma}
			\|\chi_T(t)\pi_n\psi_N^{\dag}\|_{X_{q,\infty}^{0,\gamma}(E_n)}\lesssim RT^{-\gamma+\frac{1}{q'}}N^{-(\alpha-\frac{1}{2})+\frac{2}{q}+\delta},
		\end{equation}
		and the time-localized remainder $w_N^\dagger$ satisfies
		\begin{equation}\label{eq:wn-X0b}
			\|w_N^\dagger\|_{X^{0,b}}\le R^{-1}N^{-s},
		\end{equation}
		and, for every dyadic $L\ge 2N$,
		\begin{equation}\label{eq:wn-tail-X0b}
			\|(\mathrm{Id}-P_{\le L})w_N^\dagger\|_{X^{0,b}}
			\le \left(\frac{N}{L}\right)^{10}R^{-1}N^{-s}.
		\end{equation}
	\end{proposition}
	
	\begin{remark}
		The estimate \eqref{eq:psin-X0gamma} is not written in this form in Definition 5.2 and Corollary 5.7 of \cite{Burq2024}; it is obtained from the Fourier-Lebesgue bounds in \cite[Lemma 7.2]{Burq2024}. This is the only additional piece of the RAO analysis that we shall need in Section \ref{sec:nonlinear-pointwise}.
	\end{remark}

	\section{Almost sure pointwise convergence of the NLS flow}\label{sec:nonlinear-pointwise}

	In this section, we prove Theorem \ref{thm:2}. We work with the $\mu_\alpha$-almost sure solution provided by Theorem \ref{thm:BCST}. The singular terms $\psi_N$ will be controlled in $L^\infty_{t,x}$ by combining the time-regularity estimate \eqref{eq:psin-X0gamma} with the law-equivalence lemma, while the remainders $w_N$ will be treated by Lemma \ref{maximal-easy} together with \eqref{eq:wn-X0b}--\eqref{eq:wn-tail-X0b}.
	
	Fix $R\ge R_0$, let $T=R^{-10/\theta}$ be as in Proposition \ref{prop:loc-bounds}, and define
	\begin{equation}\label{eq:TR}
		T_R:=\min\Big\{\tau_R,\frac{T}{2}\Big\}
		=\min\Big\{R^{-C_1},\frac12R^{-10/\theta}\Big\}.
	\end{equation}
	On the interval $[-T_R,T_R]$, the sequence $u_N(t)$ converges to the limit $u(t)$ from Theorem \ref{thm:BCST}, and the identities in \eqref{eq:dagger-equals-original} are valid for every dyadic $N$.
	
	\begin{proposition}\label{prop:psi-summable}
		Fix $R\ge R_0$ and set
		\[
		\epsilon_0:=\alpha-1-\frac1q-\delta>0.
		\]
		Then there exist constants $C,c,\delta_*>0$ and a measurable set $\Sigma_R'\subset\Sigma_R^{\mathrm{loc}}$ such that
		\begin{equation}\label{eq:SigmaRprime-measure}
			\mu_\alpha\big((\Sigma_R')^c\big)\le Ce^{-cR^{\delta_*}},
		\end{equation}
		and, for every $\phi\in\Sigma_R'$ and every dyadic $M$, the colored terms of the corresponding solution $u(t)$ satisfy
		\begin{equation}\label{eq:psi-series-summable}
			\sum_{N>M}\|\psi_N\|_{L_{t,x}^\infty([-T_R,T_R]\times\S^2)}\lesssim_R M^{-\epsilon_0}.
		\end{equation}
	\end{proposition}
	
	\begin{proof}
		Set
		\[
		\beta:=\gamma-\frac1{q'}>0.
		\]
		Choose $p_0>2$ so large that
		\[
		\eta_0:=\frac1q+\delta-\frac2{p_0}>0.
		\]
		Let
		\[ \delta_N:=N^{-\frac{3+\epsilon_0}{\beta}},\qquad K_N:=\Big\lceil \frac{T_R}{\delta_N}\Big\rceil,
		\qquad t_{j,N}:=j\delta_N\quad (|j|\le K_N).
		\]
		Since $T_R\le T/2$, we have $\chi_T(t_{j,N})=1$ for all $|j|\le K_N$.
		
		Let
		\[
		A_R:=R T^{-\beta}=R^{1+10\beta/\theta}.
		\]
		For each dyadic $N$, define
		\[
		\Sigma_{R,N}'
		:=
		\Big\{\phi\in \Sigma_R^{\mathrm{loc}}:
		\|\psi^{\dagger}_N(t_{j,N})\|_{L_x^{\infty}}
		\le A_RN^{-\epsilon_0}\ \text{for all }|j|\le K_N\Big\},
		\]
		and set
		\[
		\Sigma_R':=\bigcap_{N\in 2^{\mathbb N}}\Sigma_{R,N}'.
		\]
		We prove first that, for every $\phi\in\Sigma_R'$, one has
		\begin{equation}\label{eq:psiN-shell-bound}
			\|\psi_N\|_{L^\infty_{t,x}([-T_R,T_R]\times\S^2)}\lesssim_R N^{-\epsilon_0}.
		\end{equation}
		Since $\psi_N^\dagger=\psi_N$ on $[-T_R,T_R]$, it is enough to estimate $\psi_N^\dagger$.
		
		For $N/2<n\le N$, define
		\[
		F_{N,n}(t):=e^{it\lambda_n^2}\chi_T(t)\pi_n\psi_N^\dagger(t).
		\]
		By \eqref{eq:psin-X0gamma} and Lemma \ref{lem-Pre-embed},
		\begin{equation}\label{eq:FNn-holder}
			\|F_{N,n}\|_{C_t^{0,\beta}(\mathbb R;L_x^\infty)}
			\lesssim A_R N^{-(\alpha-\frac12)+\frac2q+\delta}.
		\end{equation}
		Now fix $s,t\in[-T_R,T_R]$. Since $\chi_T\equiv 1$ on $[-T_R,T_R]$, we may write
		\[
		\pi_n\psi_N^\dagger(t)=e^{-it\lambda_n^2}F_{N,n}(t),\qquad
		\pi_n\psi_N^\dagger(s)=e^{-is\lambda_n^2}F_{N,n}(s).
		\]
		Using $|e^{-i\theta}-1|\lesssim |\theta|^{\beta}$ for $0<\beta<1$, we obtain
		\begin{align*}
			\|\pi_n\psi_N^\dagger(t)-\pi_n\psi_N^\dagger(s)\|_{L_x^\infty}
			&\le \|F_{N,n}(t)-F_{N,n}(s)\|_{L_x^\infty}
			+|e^{-i(t-s)\lambda_n^2}-1|\,\|F_{N,n}(s)\|_{L_x^\infty}\\
			&\lesssim |t-s|^{\beta}(1+\lambda_n^{2\beta})\|F_{N,n}\|_{C_t^{0,\beta}(\mathbb R;L_x^\infty)}.
		\end{align*}
		Since $\lambda_n\sim N$ on the shell $N/2<n\le N$, \eqref{eq:FNn-holder} gives
		\[
		\|\pi_n\psi_N^\dagger(t)-\pi_n\psi_N^\dagger(s)\|_{L_x^\infty}
		\lesssim A_R N^{-(\alpha-\frac12)+\frac2q+\delta+2\beta}|t-s|^{\beta}.
		\]
		Summing over $n\in(N/2,N]$, we deduce
		\begin{equation}\label{eq:psiN-holder-full}
			\|\psi_N^\dagger(t)-\psi_N^\dagger(s)\|_{L_x^\infty}
			\lesssim A_R N^{\frac32-\alpha+\frac2q+\delta+2\beta}|t-s|^{\beta}.
		\end{equation}
		For any $t\in[-T_R,T_R]$, choose $j$ with $|j|\le K_N$ and $|t-t_{j,N}|\le \delta_N$. Since
		\[
		\Big(\frac32-\alpha+\frac2q+\delta+2\beta\Big)+\epsilon_0
		=\frac12+\frac1q+2\beta<3+\epsilon_0,
		\]
		we infer from \eqref{eq:psiN-holder-full} that
		\[
		\|\psi_N^\dagger(t)-\psi_N^\dagger(t_{j,N})\|_{L_x^\infty}
		\lesssim A_R N^{-\epsilon_0}.
		\]
		If $\phi\in\Sigma_R'$, then by definition of $\Sigma_{R,N}'$,
		\[
		\|\psi_N^\dagger(t_{j,N})\|_{L_x^\infty}\le A_RN^{-\epsilon_0}.
		\]
		Therefore,
		\[
		\|\psi_N(t)\|_{L_x^\infty}
		=\|\psi_N^\dagger(t)\|_{L_x^\infty}
		\lesssim A_RN^{-\epsilon_0},
		\qquad |t|\le T_R,
		\]
		which proves \eqref{eq:psiN-shell-bound}. Summing over dyadic $N>M$, we obtain
		\[
		\sum_{N>M}\|\psi_N\|_{L^\infty_{t,x}([-T_R,T_R]\times\S^2)}
		\lesssim_R \sum_{N>M}N^{-\epsilon_0}
		\lesssim_R M^{-\epsilon_0},
		\]
		that is, \eqref{eq:psi-series-summable}.
		
		It remains to prove \eqref{eq:SigmaRprime-measure}. Since $\mu_\alpha$ is the law of $\phi_\alpha^\omega$, it suffices to estimate the pullback of $\Sigma_R'$ to the Gaussian probability space. For fixed $N$ and $j$, consider the event
		\[
		E_{N,j}:=\Big\{\omega:\ \|\psi_N^\dagger(t_{j,N})\|_{L_x^\infty}>A_RN^{-\epsilon_0}\Big\}.
		\]
		By Bernstein's inequality on the frequency shell $(N/2,N]$,
		\[
		\|\psi_N^\dagger(t_{j,N})\|_{L_x^\infty}
		\lesssim N^{2/p_0}\|\psi_N^\dagger(t_{j,N})\|_{L_x^{p_0}}.
		\]
		Hence
		\begin{equation}\label{eq:ENj-reduction}
			E_{N,j}
			\subset
			\Big\{\omega:\ \|\psi_N^\dagger(t_{j,N})\|_{L_x^{p_0}}
			>cA_RN^{-\alpha+1+\eta_0}\Big\}
		\end{equation}
		for some absolute constant $c>0$. Since $|t_{j,N}|\le T_R\le T/2$, we have $\chi_T(t_{j,N})=1$; moreover, by Lemma \ref{lem-invariance}, the conditional law of $\psi_N^\dagger(t_{j,N})$ given $\mathcal B_{\le N/2}$ coincides with the law of $P_N\phi_\alpha^\omega$. Therefore,
		\[
		\mathbb P(E_{N,j})
		\le
		\mathbb P\Big(\|P_N\phi_\alpha^\omega\|_{L_x^{p_0}}>cA_RN^{-\alpha+1+\eta_0}\Big).
		\]
		Applying Corollary \ref{coro:probability of linear solution} with $p=p_0$ and
		\[
		\varepsilon_{N,R}:=\exp(-c_0A_RN^{\eta_0})
		\]
		for a sufficiently small absolute constant $c_0>0$, and enlarging the implicit constant in the definition of $A_R$ if necessary, we obtain
		\[
		\mathbb P(E_{N,j})\le \varepsilon_{N,R}=e^{-c_0A_RN^{\eta_0}}.
		\]
		Since $K_N\leq T_RN^{\frac{3+\epsilon_0}{\beta}}$, we have
		\[
		\sum_{|j|\le K_N}\mathbb P(E_{N,j})\leq T_RN^{\frac{3+\epsilon_0}{\beta}}e^{-c_0A_RN^{\eta_0}}.
		\]
		Moreover,
		\[
		(\Sigma_R')^c\subset (\Sigma_R^{\mathrm{loc}})^c\cup \bigcup_{N\in 2^{\mathbb N}}\bigcup_{|j|\le K_N}E_{N,j}.
		\]
		Therefore,
		\[
		\mu_\alpha\big((\Sigma_R')^c\big)
		\le \mu_\alpha\big((\Sigma_R^{\mathrm{loc}})^c\big)
		+\sum_{N\in 2^{\mathbb N}}\sum_{|j|\le K_N}\mathbb P(E_{N,j})
		\lesssim e^{-c_2R^{\delta_1}}+\sum_{N\in 2^{\mathbb N}}T_RN^{\frac{3+\epsilon_0}{\beta}}e^{-c_0A_RN^{\eta_0}}.
		\]
		Since $A_R=R^{1+10\beta/\theta}$ is a positive power of $R$, the dyadic series on the right-hand side is bounded by $Ce^{-cA_R}$, and hence
		\[
		\mu_\alpha\big((\Sigma_R')^c\big)
		\le C e^{-cR^{\delta_*}}
		\]
		for some constants $C,c,\delta_*>0$. This completes the proof of Proposition \ref{prop:psi-summable}.
	\end{proof}
	
	\begin{lemma}\label{lem:w-summable}
		Fix $R\ge R_0$ and choose $s_0\in (\frac12,s)$. Then, for every initial data $\phi\in\Sigma_R^{\mathrm{loc}}$, the nonlinear remainders $w_N$ for the solution $u(t)$ with initial data $\phi$ in Theorem \ref{thm:BCST} satisfy the tail estimate
		\begin{equation}\label{eq:w-series-summable}
			\sum_{N>M}\|w_N\|_{L_x^2L_t^\infty([-T_R,T_R]\times\S^2)}\lesssim_R M^{-(s-s_0)}.
		\end{equation}
	\end{lemma}
	
	\begin{proof}
		Let $\widetilde\chi\in C_c^\infty(\mathbb{R})$ be such that $\widetilde\chi\equiv 1$ on $[-T,T]$. Since $w_N^\dagger$ is compactly supported in $[-T,T]$, Lemma \ref{maximal-easy} gives
		\[
		\|w_N\|_{L_x^2L_t^\infty([-T_R,T_R]\times\S^2)}
		\le
		\|w_N^\dagger\|_{L_x^2L_t^\infty(\mathbb{R}\times\S^2)}
		=\|\widetilde\chi w_N^\dagger\|_{L_x^2L_t^\infty}
		\lesssim \|w_N^\dagger\|_{X^{s_0,b}}.
		\]
		We now estimate the right-hand side by using \eqref{eq:wn-X0b} and \eqref{eq:wn-tail-X0b}. Decompose dyadically as
		\[
		\|w_N^\dagger\|_{X^{s_0,b}}
		\lesssim
		\|P_{\le 2N}w_N^\dagger\|_{X^{s_0,b}}+
		\sum_{L\in 2^{\mathbb{N}},\,L\ge 4N}\|P_Lw_N^\dagger\|_{X^{s_0,b}}.
		\]
		For the low-frequency part, \eqref{eq:wn-X0b} yields
		\[
		\|P_{\le 2N}w_N^\dagger\|_{X^{s_0,b}}
		\lesssim N^{s_0}\|w_N^\dagger\|_{X^{0,b}}
		\lesssim_R N^{s_0-s}.
		\]
		For $L\ge 4N$, we have $P_Lw_N^\dagger=P_L(\mathrm{Id}-P_{\le L/2})w_N^\dagger$, and hence
		\begin{align*}
			\|P_Lw_N^\dagger\|_{X^{s_0,b}}
			&\lesssim L^{s_0}\|P_Lw_N^\dagger\|_{X^{0,b}}\\
			&\le L^{s_0}\|(\mathrm{Id}-P_{\le L/2})w_N^\dagger\|_{X^{0,b}}\\
			&\lesssim_R L^{s_0}\left(\frac{N}{L/2}\right)^{10}N^{-s}
			\lesssim_R N^{10-s}L^{s_0-10}.
		\end{align*}
		Since $s_0<10$, the dyadic sum over $L\ge 4N$ is bounded by a constant multiple of $N^{s_0-s}$. Therefore,
		\[
		\|w_N^\dagger\|_{X^{s_0,b}}\lesssim_R N^{s_0-s}.
		\]
		Because $s_0-s<0$, the dyadic series $\sum_N N^{s_0-s}$ converges, and \eqref{eq:w-series-summable} follows.
	\end{proof}
	
	\begin{proposition}\label{prop:nonlinear-tail}
		For every $R\ge R_0$, let
		\[
		\Sigma_R^{''}:=\Sigma_R\cap\Sigma_R'.
		\]
		Then
		\begin{equation}\label{eq:OmegaR-estimate}
			\mu_{\alpha}((\Sigma_R^{''})^c)\le Ce^{-cR^{\delta_*}}
		\end{equation}
		for some constants $C,c,\delta_*>0$, and for every $\phi\in\Sigma_R^{''}$, the solutions $u_M(t)$ of \eqref{fml-WickOrd-NLS} with initial data $P_{\le M}\phi$ and their limit $u(t)$ satisfy
		\begin{equation}\label{eq:nonlinear-tail-convergence}
			\lim_{M\to\infty}\left\|\sup_{|t|\le T_R}|u(t)-u_M(t)|\right\|_{L_x^2(\S^2)}=0.
		\end{equation}
	\end{proposition}
	
	\begin{proof}
		The measure estimate \eqref{eq:OmegaR-estimate} follows immediately from Theorem \ref{thm:BCST} and Proposition \ref{prop:psi-summable}. Now fix $\phi\in\Sigma_R^{''}$. On $[-T_R,T_R]$, the limit solution $u$ from Theorem \ref{thm:BCST} satisfies
		\[
		u(t)-u_M(t)=\sum_{N>M}(\psi_N(t)+w_N(t)).
		\]
		Therefore,
		\begin{align*}
			\left\|\sup_{|t|\le T_R}|u(t)-u_M(t)|\right\|_{L_x^2}
			\le& \left\|\sup_{|t|\le T_R}\left|\sum_{N>M}\psi_N(t)\right|\right\|_{L_x^2}
			+\left\|\sup_{|t|\le T_R}\left|\sum_{N>M}w_N(t)\right|\right\|_{L_x^2}\\
			\lesssim& \sum_{N>M}\|\psi_N\|_{L^\infty([-T_R,T_R]\times\S^2)}\\
			&+\sum_{N>M}\|w_N\|_{L_x^2L_t^\infty([-T_R,T_R]\times\S^2)}.
		\end{align*}
		The first term tends to $0$ by Proposition \ref{prop:psi-summable}. The second one tends to $0$ by Lemma \ref{lem:w-summable}. Hence \eqref{eq:nonlinear-tail-convergence} follows.
	\end{proof}
	
	\begin{proof}[Proof of Theorem \ref{thm:2}]
		Set
		\[
		\Sigma_*:=\bigcup_{m\ge R_0,\,m\in\mathbb{N}}\Sigma_m^{''}.
		\]
		By \eqref{eq:OmegaR-estimate},
		\[
		\mu_{\alpha}((\Sigma_*)^c)=\mu_{\alpha}\Big(\bigcap_{m\ge R_0}(\Sigma_m^{''})^c\Big)
		\le \inf_{m\ge R_0}\mu_{\alpha}((\Sigma_m^{''})^c)=0.
		\]
		Hence $\mu_{\alpha}(\Sigma_*)=1$.
		
		Fix $\phi\in\Sigma_*$. Then there exists $m\ge R_0$ such that $\phi\in\Sigma_m^{''}$. Let $u$ be the corresponding local solution on $[-T_m,T_m]$, where $T_m$ denotes the quantity in \eqref{eq:TR} with $R=m$. For $\varepsilon>0$, define
		\[
		E_\varepsilon:=\Big\{x\in\S^2: \limsup_{t\to 0}|u(t,x)-\phi(x)|>\varepsilon\Big\}.
		\]
		We remark that $E_{\epsilon}$ depends on $\phi$. Below we will prove that $E_{\epsilon}$ is a null Lebesgue measurable set.
		
		To see this, fix a dyadic number $M$. Since $u_M$ is a smooth solution with smooth initial data $P_{\le M}\phi$, we have
		\[
		\lim_{t\to 0}u_M(t,x)=P_{\le M}\phi(x)
		\qquad\text{for every }x\in\S^2.
		\]
		Therefore,
		\begin{align*}
			E_\varepsilon
			\subset&~
			\Big\{x\in\S^2: \sup_{|t|\le T_m}|u(t,x)-u_M(t,x)|>\frac{\varepsilon}{2}\Big\}\\
			&\cup
			\Big\{x\in\S^2: |P_{>M}\phi(x)|>\frac{\varepsilon}{2}\Big\}.
		\end{align*}
		By Chebyshev's inequality and Proposition \ref{prop:nonlinear-tail},
		\begin{align*}
			|E_{\varepsilon}|
			&\le \frac{4}{\varepsilon^2}
			\left\|\sup_{|t|\le T_m}|u(t)-u_M(t)|\right\|_{L_x^2}^2
			+\frac{4}{\varepsilon^2}\|P_{>M}\phi\|_{L_x^2}^2\\
			&\longrightarrow 0
			\qquad\text{as }M\to\infty.
		\end{align*}
		Here the first term tends to $0$ by Proposition \ref{prop:nonlinear-tail}, while the second tends to $0$ because $\phi\in\Sigma_m\subset H^{(\alpha-1)-}(\S^2)\subset L^2(\S^2)$ and $P_{>M}\phi\to 0$ in $L^2$. We conclude that $|E_\varepsilon|=0$ for every $\varepsilon>0$, that is,
		\[
		\lim_{t\to 0}u(t,x)=\phi(x)
		\qquad\text{for a.e. }x\in\S^2.
		\]
		This proves Theorem \ref{thm:2}.
	\end{proof}

\end{document}